\documentclass[11pt]{amsart}
\usepackage{eurosym}
\usepackage[pagebackref,colorlinks=true]{hyperref}
\usepackage{verbatim}
\usepackage{eucal,url,amssymb,stmaryrd,enumerate,amscd,}
\usepackage{amsfonts}

\usepackage{amsmath,amsthm,amssymb,amscd,enumerate,eucal,url,stmaryrd}
\usepackage[margin=1in]{geometry}

\numberwithin{equation}{section}
\newtheorem{thrm}{Theorem}[section]
\newtheorem{lemma}[thrm]{Lemma}
\newtheorem{prop}[thrm]{Proposition}
\newtheorem{cor}[thrm]{Corollary}

\allowdisplaybreaks

\begin{document}

\begin{abstract}
New smooth solutions of the Strominger system with non vanishing flux, non-trivial instanton
and non-constant dilaton based on the quaternionic Heisenberg group are constructed.
We show that through appropriate contractions
the solutions found in the $G_2$-heterotic case
converge to the heterotic solutions on 6-dimensional inner non-K\"ahler
spaces previously found by the authors
and, moreover, to new heterotic solutions with non-constant dilaton in dimension 5.
All the solutions satisfy the heterotic
equations of motion up to the first order of $\alpha^{\prime}$.
\end{abstract}

\title[Heterotic String Solutions with non-constant dilaton in dimensions 7
and 5]{The quaternionic Heisenberg group and Heterotic String Solutions with
non-constant dilaton in dimensions 7 and 5}
\date{\today}
\author{Marisa Fern\'andez}
\address[Fern\'andez]{Universidad del Pa\'{\i}s Vasco\\
Facultad de Ciencia y Tecnolog\'{\i}a, Departamento de Matem\'aticas\\
Apartado 644, 48080 Bilbao\\
Spain}
\email{marisa.fernandez@ehu.es}
\author{Stefan Ivanov}
\address[Ivanov]{University of Sofia "St. Kl. Ohridski"\\
Faculty of Mathematics and Informatics\\
Blvd. James Bourchier 5\\
1164 Sofia, Bulgaria}
\address{and Institute of Mathematics and Informatics, Bulgarian Academy of
Sciences}
\email{ivanovsp@fmi.uni-sofia.bg}
\author{Luis Ugarte}
\address[Ugarte]{Departamento de Matem\'aticas\,-\,I.U.M.A.\\
Universidad de Zaragoza\\
Campus Plaza San Francisco\\
50009 Zaragoza, Spain}
\email{ugarte@unizar.es}
\author{Dimiter Vassilev}
\address[Dimiter Vassilev]{ Department of Mathematics and Statistics\\
University of New Mexico\\
Albuquerque, New Mexico, 87131-0001}
\email{vassilev@math.unm.edu}
\date{\today }
\maketitle
\tableofcontents

\setcounter{tocdepth}{2}

\section{Introduction}

We investigate smooth solutions with non-trivial fluxes to the heterotic
equations of motion preserving at least one supersymmetry up to the first
order of the string tension $\alpha^{\prime }$ in dimensions seven and five.
Using the quaternionic Heisenberg group we propose an explicit construction
leading to new smooth solutions of the Killing spinor equations and the
Green-Schwarz anomaly cancellation, the system of equations known as the
Strominger system, with a non-constant dilaton.
The found solutions satisfy the heterotic
equations of motion up to the first order of $\alpha^{\prime }$.

Another goal of the paper is to point that through contractions of the
quaternion Heisenberg algebra, the geometric structures, the partial
differential equations and their solutions found in the $G_2$-heterotic case
converge to the heterotic solutions on 6-dimensional inner non-K\"ahler
spaces found in \cite{FIUVas} and to the new 5-dimensional heterotic solutions with non-constant dilaton.

The bosonic fields of the ten-dimensional supergravity which arises as low
energy effective theory of the heterotic string are the spacetime metric $g$%
, the NS three-form field strength (flux) $H$, the dilaton $\phi$ and the
gauge connection $A$ with curvature 2-form $F^A$. The bosonic geometry is of
the form $\mathbb{R}^{1,9-d}\times M^d$, where the bosonic fields are
non-trivial only on $M^d$, $d\leq 8$. We consider the two connections $
\nabla^{\pm}=\nabla^g \pm \frac12 H, $ 
where $\nabla^g$ is the Levi-Civita connection of the Riemannian metric~$g$.
Both connections preserve the metric, $\nabla^{\pm}g=0$ and have totally
skew-symmetric torsion $\pm H$, respectively. We denote by $R^g,R^{\pm}$ the
corresponding curvature.

We consider the heterotic supergravity theory with an $\alpha^{\prime }$
expansion where $1/2\pi\alpha^{\prime }$ is the heterotic string tension.
The bosonic part of the ten-dimensional supergravity action in the string
frame is (\cite{HT}, \cite{Berg}, $R=R^-$)
\begin{gather}  \label{action}
S=\frac{1}{2k^2}\int d^{10}x\sqrt{-g}e^{-2\phi}\Big[Scal^g+4(\nabla^g%
\phi)^2- \frac{1}{2}|H|^2 -\frac{\alpha^{\prime }}4\Big(Tr |F^A|^2)-Tr |R|^2%
\Big)\Big].
\end{gather}
The string frame field equations (the equations of motion induced from the
action \eqref{action}) of the heterotic string up to the first order of $%
\alpha^{\prime }$ in sigma model perturbation theory in the notations in
\cite{GPap} are \cite{Hu86,HT}

\begin{equation}  \label{mot}
\begin{aligned}
&Ric^g_{ij}-\frac14H_{imn}H_j^{mn}+2\nabla^g_i\nabla^g_j\phi-\frac{\alpha^{%
\prime }}4 \Big[(F^A)_{imab}(F^A)_j^{mab}-R_{imnq}R_j^{mnq}\Big]=0, \\
&\nabla^g_i(e^{-2\phi}H^i_{jk})=0,\qquad \nabla^+_i(e^{-2\phi}(F^A)^i_j)=0.
\end{aligned}
\end{equation}
The field equation of the dilaton $\phi$ is implied from the first two
equations above.

The Green-Schwarz anomaly cancellation mechanism requires that the
three-form Bianchi identity receives an $\alpha^{\prime }$ correction of the
form
\begin{equation}  \label{acgen}
dH=\frac{\alpha^{\prime }}48\pi^2(p_1(M^d)-p_1(E))=\frac{\alpha^{\prime }}4 %
\Big(Tr(R\wedge R)-Tr(F^A\wedge F^A)\Big),
\end{equation}
where $p_1(M^d)$ and $p_1(E)$ are the first Pontrjagin forms of $M^d$ with
respect to a connection $\nabla$ with curvature $R$ and the vector bundle $E$
with connection $A$, respectively.

A class of heterotic-string backgrounds for which the Bianchi identity of
the three-form $H$ receives a correction of type \eqref{acgen} are those
with (2,0) world-volume supersymmetry. Such models were considered in \cite%
{HuW}. The target-space geometry of (2,0)-supersymmetric sigma models has
been extensively investigated in \cite{HuW,Str,HP1}. Recently, there is
revived interest in these models \cite%
{Bwit,GKMW,CCDLMZ,IP1,IP2,GMPW,GMW,GPap} as string backgrounds and in
connection with heterotic-string compactification with fluxes mainly in
dimension six \cite%
{Car1,BBDG,BBE,BBDP,y1,FLY,y3,y4,P,GPR,BBCG,GPRS,GLP,Pap,BSethi,AG1,AG2,AG3,GKP,Bis,Sethi,Ossa,Sharp,MSeth,FYau,Ossa1}%
.

Equations \eqref{acgen}, \eqref{action} and \eqref{mot} involve a subtlety
due to the choice of the connection $\nabla$ on $TM^d$ since anomalies can
be canceled independently of the choice \cite{Hull}. Different connections
correspond to different regularization schemes in the two-dimensional
worldsheet non-linear sigma model. Hence the background fields given for the
particular choice of $\nabla$ must be related to those for a different
choice by a field redefinition \cite{Sen}. Connections on $M^d$ proposed to
investigate the anomaly cancellation \eqref{acgen} are $\nabla^g$ \cite%
{Str,GMW}, $\nabla^+$ \cite{CCDLMZ,DFG,FIUV}, $\nabla^-$ \cite%
{Hull,Berg,Car1,GPap,II,KY,KM,MS,MY,IMY}, Chern connection $\nabla^c$ when $%
d=6$ \cite{Str,y1,FLY,y3,y4}.

A heterotic geometry preserves supersymmetry iff in ten
dimensions there exists at least one Majorana-Weyl spinor $\epsilon$ such
that the following Killing-spinor equations hold \cite{Str,Berg}
\begin{equation} \label{sup1}
\begin{aligned}
&\delta_{\lambda}=\nabla_m\epsilon = \left(\nabla_m^g +\frac{1}{4}%
H_{mnp}\Gamma^{np} \right)\epsilon=\nabla^+\epsilon=0, \\
& \delta_{\Psi}=\left(\Gamma^m\partial_m\phi -\frac{1}{12}H_{mnp}\Gamma^{mnp}
\right)\epsilon=(d\phi-\frac12H)\cdot\epsilon=0, \\
& \delta_{\xi}=F^A_{mn}\Gamma^{mn}\epsilon=F^A\cdot\epsilon=0,
\end{aligned}
\end{equation}
where $\lambda, \Psi, \xi$ are the gravitino, the dilatino and the gaugino
fields,  $\Gamma_i$ generate the Clifford algebra $%
\{\Gamma_i,\Gamma_j\}=2g_{ij}$ and $\cdot$ means Clifford action of forms on
spinors.

The system of Killing spinor equations \eqref{sup1} together with the
anomaly cancellation condition \eqref{acgen} is known as the \emph{%
Strominger system} \cite{Str}. The last equation in \eqref{sup1} is the
instanton condition which means that the curvature $F^A$ is contained in a
Lie algebra of a Lie group which is a stabilizer of a non-trivial spinor. In
dimension 7 this group is $G_2$. Denoting the $G_2$ three-form by $\Theta$,
the $G_2$-instanton condition has the form
\begin{equation}  \label{in2}
\sum_{k,l=1}^7(F^A)^i_j(E_k,E_l)\Theta(E_k,E_l,E_m)=0.
\end{equation}

In the presence of a curvature term $Tr(R\wedge R)$ the solutions of the
Strominger system \eqref{sup1}, \eqref{acgen} obey the second and the third
equations of motion (the second and the third equations in \eqref{mot}) but
do not always satisfy the Einstein equations of motion (see \cite%
{FIUV,FIUVdim5,FIUVdim7-8} where a sufficient quadratic condition on $R$ is
found). It was proved in \cite{Iv0} that the solutions of the Strominger
system (\eqref{sup1} and \eqref{acgen}) also solve the heterotic
supersymmetric equations of motion \eqref{mot} if and only if $R$ is an
instanton in dimensions 5,6,7,8 (see \cite{MS,Ossa} for higher dimensions
and different proofs). In particular, in dimension 7, $R$ is required to be
an $G_2$-instanton.

The physically relevant connection on the tangent bundle to be considered in %
\eqref{acgen}, \eqref{action}, \eqref{mot} is the $(-)$-connection \cite%
{Berg,Hull}. One reason is that the curvature $R^-$ of the $(-)$-connection
is an instanton up to the first order of $\alpha^{\prime }$ which is a
consequence of the first equation in \eqref{sup1}, \eqref{acgen} and the
well known identity
\begin{equation}  \label{dtr}
R^+(X,Y,Z,U)-R^-(Z,U,X,Y)=\frac12dH(X,Y,Z,U).
\end{equation}
Indeed, \eqref{acgen} together with \eqref{dtr} imply $%
R^+(X,Y,Z,U)-R^-(Z,U,X,Y)=O(\alpha^{\prime })$ and the first equation in %
\eqref{sup1} yields that the holonomy group of $\nabla^+$ is contained in $%
G_2$, i.e. the curvature 2-form $R^+(X,Y)\subset \mathfrak{g}_2$ and
therefore $R^-$ satisfies the instanton condition \eqref{in2} up to the
first order of $\alpha^{\prime }$. Hence, a solution to the Strominger
system with first Pontrjagin form of the $(-)$-connection always satisfies
the heterotic equations of motion \eqref{mot} up to the first order of $%
\alpha^{\prime }$ (see e.g.\cite{MS} and references therein).

In dimension 7 the only known heterotic/type I solutions with non-zero
fluxes to the equations of motion preserving at least one supersymmetry
(satisfying \eqref{sup1} and \eqref{acgen} without the curvature term, $R=0$%
) are those constructed \cite{GNic}. All these solutions are noncompact and
conformal to a flat space. Noncompact solutions to \eqref{sup1} and %
\eqref{acgen} in dimension 7 are presented also in \cite{II}. The first
compact heterotic/type I solutions with non-zero fluxes and constant dilaton
to the equations of motion preserving at least one supersymmetry (satisfying %
\eqref{sup1} and \eqref{acgen}) in dimension seven are constructed in \cite%
{FIUVdim7-8}

In dimension $d=5$, if the field strength vanishes, $H=0$, then the
5-dimensional case reduces to dimension four since any five dimensional
Riemannian spin manifold admitting $\nabla^g$-parallel spinor is reducible.
Non compact solutions on circle bundle over 4-dimensional base endowed with
a hyper K\"ahler metric (when the 4-dimensional metric is Eguchi-Hanson,
Taub-NUT, Atiyah-Hitchin) have appeared in \cite{LV-P,GGMPR,SM,BBW,Pap}, the
compact cases are discussed in \cite{GMW} where a cohomological obstruction
is presented. The first compact heterotic/type I solutions with non-zero
fluxes and constant dilaton to the equations of motion preserving at least
one supersymmetry (satisfying \eqref{sup1} and \eqref{acgen}) in dimension
five are constructed in \cite{FIUVdim5}.

In this paper we construct smooth solutions with non vanishing flux and
non-constant dilaton to the Strominger system using the first Pontrjagin
form of the \emph{$(-)$-connection} on 7-dimensional complete non-compact
manifold equipped with conformally cocalibrated $G_2$ structures of pure
type coupled with carefully chosen instanton bundle. The source of the
construction is the already constructed smooth compact solutions to the
Strominger system with constant dilaton on nilmanifods presented in \cite%
{FIUVdim7-8} and the ideas outlined there to consider special three-torus
bundles over either conformally $\mathbb{T}^4$ or K3 manifold.

Our first family of solutions are complete $G_2$ manifolds which are $%
\mathbb{T}^3$ bundles over conformally compact asymptotically hyperbolic
metric on $\mathbb{T}^4$ with conformal boundary at infinity a flat torus $%
\mathbb{T}^3$. Using the first Pontrjagin form of the $(-)$-connection
together with the first Pontrjagin form of a carefully chosen instanton we
satisfied the anomaly cancellation condition with a \emph{negative} $%
\alpha^{\prime }$ and a non-constant dilaton which a real slice of an
elliptic function of order two.

In Section \ref{s:non-compact ex} we present another smooth non-compact
complete solution to the Strominger system using the first Pontrjagin form
of the $(-)$-connection with positive string tension on certain $\mathbb{T}%
^3 $ bundles over $\mathbb{R}^4$ with non-vanishing torsion, non-trivial
instanton and non-constant dilaton. The non-constant dilaton function here is determined by the fundamental
solution of the Laplacian on $\mathbb{R}^4$.

\medskip

\noindent\textbf{Conventions.} The connection 1-forms $\omega_{ji}$ of a metric
connection $\nabla, \nabla g=0$ with respect to a local orthonormal basis $%
\{E_1,\ldots,E_d\}$ are given by $\omega_{ji}(E_k) = g(\nabla_{E_k}E_j,E_i)$%
, since we write $\nabla_X E_j = \omega^s_j(X)\, E_s$.

The curvature 2-forms $\Omega^i_j$ of $\nabla$ are given in terms of the
connection 1-forms $\omega^i_j$ by $
\Omega^i_j = d \omega^i_j + \omega^i_k\wedge\omega^k_j, \quad \Omega_{ji} =
d \omega_{ji} + \omega_{ki}\wedge\omega_{jk}, \quad
R^l_{ijk}=\Omega^l_k(E_i,E_j), \quad R_{ijkl}=R^s_{ijk}g_{ls}$.

The first Pontrjagin class is represented by the 4-form $8\pi^2
p_1(\nabla)=\sum_{1\leq i<j\leq d} \Omega^i_j\wedge\Omega^i_j$.

\medskip

\noindent\textbf{Acknowledgments.} The work was partially supported through Project
MICINN (Spain) MTM2011-28326-C02-01/02, and Project of UPV/EHU ref.\
UFI11/52. S.I. is partially supported by Contract 168/2014 with the Sofia
University "St.Kl.Ohridski". D.V. was partially supported by the Simons Foundation grant \#279381.
S.I. and D.V. thank the University of Zaragoza and the University of the Basque Country for the hospitality
and financial support provided while visiting the respective Departments of Mathematics.

\section{The supersymmetry equations and the geometric model}

\label{geomod}

Geometrically, the vanishing of the gravitino variation is equivalent to the
existence of a non-trivial real spinor parallel with respect to the metric
connection $\nabla^+$ with totally skew-symmetric torsion $T=H$. The
presence of $\nabla^+$-parallel spinor leads to restriction of the holonomy
group $Hol(\nabla^+)$ of the torsion connection $\nabla^+$.

\subsection{Dimension 7}

\label{ss:dim7 constr}

In dimension seven $Hol(\nabla^+)$ has to be contained in the exceptional
group $G_2$ \cite{FI1,GKMW,GMW,FI2}. The precise conditions to have a
solution to the gravitino Killing spinor equation in dimension 7 were found
in \cite{FI1}. Namely, there exists a non-trivial parallel spinor with
respect to a $G_2$-connection with torsion 3-form $T$ if and only if there
exists an integrable $G_2$-structure $\Theta$, i.e. $d*\Theta=\theta^7\wedge
*\Theta$, where $\theta^7=-\frac{1}{3}*(* d\Theta\wedge\Theta) = \frac{1}{3}%
*(* d*\Theta\wedge*\Theta)$ is the Lee form. In this case, the connection $%
\nabla^+$ is unique and the torsion 3-form $T$ is given by the formula \cite%
{FI1}
\begin{equation*}
H=T=\frac{1}{6}(d\Theta,*\Theta)\,\Theta - *d\Theta +*(\theta^7\wedge\Theta).
\end{equation*}

The necessary conditions to have a solution to the system of dilatino and
gravitino Killing spinor equations (the first two equations in \eqref{sup1})
in dimension seven were derived in \cite{GKMW,FI1,FI2}, and the sufficiency
was proved in \cite{FI1,FI2}. The general existence result \cite{FI1,FI2}
states that there exists a non-trivial solution to both dilatino and
gravitino Killing spinor equations (the first two equations in \eqref{sup1})
in dimension 7 if and only if there exists a globally conformal co-calibrated $G_2$-structure $%
(\Theta,g)$ of pure type and with exact Lee form $\theta^7$, i.e. a $G_2$%
-structure $(\Theta,g)$ satisfying the equations
\begin{equation}  \label{sol7}
d*\Theta=\theta^7\wedge *\Theta, \quad d\Theta\wedge\Theta=0, \quad
\theta^7=-2d\phi.
\end{equation}
Consequently, the torsion 3-form (the flux $H$) is given by $
H=T= -* d\Theta - 2*(d\phi\wedge\Theta)$ and the Riemannian scalar curvature
satisfies $s^g=8||d\phi||^2 -\frac{1}{12}||T||^2 -6\, \delta d\phi$. The equations \eqref{sol7} hold exactly when the $G_2$-structure $%
(\bar\Theta=e^{-\frac32\phi}\Theta,\bar g=e^{-\phi}g)$ obeys the equations $d%
\bar{*}\bar\Theta=d\bar\Theta\wedge\bar\Theta=0,$ i.e., it is co-calibrated of
pure type.

A geometric model which fits the above structures  was proposed in \cite%
{FIUVdim7-8} as a certain ${\mathbb{T}}^3$-bundle over a Calabi-Yau surface.
For this, let $\Gamma_i$, $1\leq i \leq 3$, be three closed anti-self-dual $2$-forms
on a Calabi-Yau surface $M^4$, which represent integral cohomology classes.
Denote by $\omega_1$ and by $\omega_2+\sqrt{-1}\omega_3$ the (closed)
K\"ahler form and the holomorphic volume form on $M^4$, respectively. Then,
there is a compact 7-dimensional manifold $M^{1,1,1}$ which is the total
space of a ${\mathbb{T}}^3$-bundle over $M^4$ and has a $G_2$-structure
\begin{equation*}
\Theta=\omega_1\wedge\eta_1+\omega_2\wedge\eta_2-\omega_3\wedge\eta_3+\eta_1%
\wedge \eta_2\wedge\eta_3,
\end{equation*}
solving the first two Killing spinor equations in \eqref{sup1} with constant
dilaton in dimension $7$, where $\eta_i$, $1\leq i \leq 3$, is a $1$-form on

$M^{1,1,1}$ such that $d\eta_i=\Gamma_i$, $1\leq i \leq 3$.

For any smooth function $f$ on $M^4$, the $G_2$-structure on $M^{1,1,1}$
given by
\begin{equation*}
\Theta_f=e^{2f}\Big[\omega_1\wedge\eta_1+\omega_2\wedge\eta_2-
\omega_3\wedge\eta_3\Big]+\eta_1\wedge\eta_2\wedge\eta_3
\end{equation*}
solves the first two Killing spinor equations in \eqref{sup1} with
non-constant dilaton $\phi=-2f$. The metric has the form
\begin{equation*}
g_f=e^{2f}g_{cy}+\eta_1\otimes\eta_1+\eta_2\otimes\eta_2+
\eta_3\otimes\eta_3.
\end{equation*}
To achieve a smooth solution to the Strominger system we still have to
determine an auxiliary vector bundle with an instanton and a linear
connection on $M^{1,1,1}$ 
in order to satisfy the anomaly cancellation condition \eqref{acgen}.

\subsection{Dimension 5}

The existence of $\nabla^+$-parallel spinor in dimension 5 determines an
almost contact metric structure and, equivalently, a reduction of the
structure group $SO(5)$ to $SU(2)$. The properties of the almost contact
metric structure as well as solutions to gravitino and dilatino
Killing-spinor equations are investigated in \cite{FI1,FI2} and presented in
terms of reduction to $SU(2)$ in \cite{FIUVdim5}.

\subsubsection{Almost contact structure point of view}

We recall that an almost contact metric structure consists of an odd
dimensional manifold $M^{2k+1}$ equipped with a Riemannian metric $g$,
vector field $\xi$ of length one, its dual 1-form $\eta$ as well as an
endomorphism $\psi$ of the tangent bundle such that
\begin{equation*}
\psi(\xi)=0, \quad \psi^2=-id +\eta\otimes\xi, \quad
g(\psi.,\psi.)=g(.,.)-\eta\otimes\eta.
\end{equation*}
The Reeb vector field $\xi$ is determined by the equations $%
\eta(\xi)=1,\quad \xi\lrcorner d\eta=0$, where $\lrcorner$ denotes the
interior multiplication. The Nijenhuis tensor $N$, the fundamental form $F$
and the Lee form $\theta$ of an almost contact metric structure are defined
by
\begin{equation*}
N=[\psi.,\psi.]+\psi^2[.,.] -\psi [\psi.,.]-\psi[.,\psi.] +d\eta\otimes\xi,
\quad F(.,.)=g(.,\psi.), \quad \theta=\frac12F\lrcorner dF.
\end{equation*}
It was shown in \cite{FI2} that the gravitino and the dilatino equation
admit a solution in dimension five if and only if the Nijenhuis tensor is
totally skew-symmetric, the Reeb vector field $\xi$ is a Killing vector
field and the next equalities hold $2d\phi=\theta, \quad *_{\mathbb{H}%
}d\eta=-d\eta$, where $*_{\mathbb{H}}$ denotes the Hodge operator acting on
the 4-dimensional orthogonal complement $\mathbb{H}$ of the vector $\xi$, $%
\mathbb{H}={\rm Ker}\, \eta$.

\subsubsection{The SU(2)-structure point of view.}

The reduction of the structure group $SO(5)$ to $SU(2)$ is described in
terms of forms by Conti and Salamon in \cite{ConS} (see also \cite{GGMPR})
as follows: an $SU(2)$-structure on a 5-dimensional manifold $M$ is $%
(\eta,F=\omega_1,\omega_2,\omega_3)$, where $\eta$ is a $1$-form dual to $%
\xi $ via the metric and $\omega_s$, $s=1,2,3$, are $2$-forms on $M$ satisfying
$\omega_s\wedge \omega_t=\delta_{st}\, v, \quad v\wedge\eta\not=0$ for some $4$%
-form $v$, and $X\lrcorner \omega_1=Y\lrcorner \omega_2\Rightarrow
\omega_3(X,Y)\ge 0$. The 2-forms $\omega_s$, $s=1,2,3$, can be chosen to form a
basis of the $\mathbb{H}$-self-dual 2-forms \cite{ConS}.

It was shown in \cite{FIUVdim5} that the first two equations in \eqref{sup1}
admit a solution in dimension five exactly when there exists a five
dimensional manifold $M$ endowed with an $SU(2)$-structure
$(\eta,F=\omega_1,\omega_2,\omega_3$) satisfying the structure equations:
\begin{equation}  \label{solstr1}
d\omega_s=2df\wedge \omega_s, \qquad *_{\mathbb{H}}d\eta = - d\eta, \qquad
df(\xi)=0.
\end{equation}
The flux $H$ is given by \cite{FI1,FI2}
\begin{equation}  \label{tor5f}
H=T=\eta\wedge d\eta +2d^{\psi}f\wedge F, \quad \text{where} \quad
d^{\psi}f(X)=-df(\psi X).
\end{equation}
The dilaton $\phi$ is equal to $\phi=2f$.

In other words, the gravitino and dilatino equations in dimension five are
satisfied if and only if the manifold is special conformal to a
quasi-Sasaki manifold with $\mathbb{H}$-anti-self-dual exterior derivative
of the almost contact form and the metric has the form
\begin{equation*}
g_f=e^{2f}g_{|_{\mathbb{H}}}+\eta\otimes\eta.
\end{equation*}

It was proposed in \cite{FIUVdim7-8} to investigate $S^1$ bundles over a
conformally hyper-K\"ahler manifold. This ansatz guaranties solution to the
first two equations in \eqref{sup1}. To achieve a smooth solution to the
Strominger system we still have to determine  a linear connection on the tangent bundle and an auxiliary vector bundle with
an $SU(2)$-instanton, i.e., a connection $A$ with curvature 2-form $F^A$
satisfying
\begin{equation}  \label{25}
(F^A)^i_j(\psi E_k,\psi E_l)=(F^A)^i_j(E_k,E_l),\qquad
\sum_{k=1}^5(F^A)^i_j(E_k,\psi E_k)=0
\end{equation}
so that the
anomaly cancellation condition \eqref{acgen} is satisfied.

\section{The quaternionic Heisenberg group}

\label{calculations}

The seven dimensional quaternionic Heisenberg group $G(\mathbb{H})$ is the
connected simply connected Lie group with a group multiplication $[.,.]$
determined by the Lie algebra $\mathfrak{g(\mathbb{H})}$ with structure
equations
\begin{equation}  \label{ecus-qHg}
\begin{aligned} & d\gamma^1=d\gamma^2=d\gamma^3=d\gamma^4=0, \quad
d\gamma^5=\gamma^{12}-\gamma^{34},\quad
d\gamma^6=\gamma^{13}+\gamma^{24},\quad d\gamma^7=\gamma^{14}-\gamma^{23}.
\end{aligned}
\end{equation}

In order to obtain results in dimensions less than seven through
contractions of $\mathfrak{g(\mathbb{H})}$ it will be convenient to consider
the orbit of $G(\mathbb{H})$ under the natural action of $GL(3,\mathbb{R})$
on the $span\, \{\gamma^5, \gamma^6, \gamma^7\}$. Accordingly let $K_A$ be a
seven-dimensional real Lie group with Lie bracket $[x,x^{\prime
}]_A=A[A^{-1}x,A^{-1}x^{\prime }]$ for $A\in GL(3,\mathbb{R})$ defined by a
basis of left-invariant 1-forms $\{e^1,\ldots,e^7\}$ such that $e^i=\gamma^i$
for $1 \leq i \leq 4$ and $(e^5\ e^6\ e^7)=A\, (\gamma^5\ \gamma^6\
\gamma^7)^T$.
Hence, the structure equations of the Lie algebra $\mathfrak{K}_A$ of the
group $K_A$ are
\begin{equation}  \label{ecus-general}
d e^1=d e^2=d e^3=d e^4=0, \qquad d e^{4+i}= \sum_{j=1}^3a_{ij}\,\sigma_j,
\quad i=1,2,3,
\end{equation}
where $\sigma_1=e^{12}-e^{34}$, $\sigma_2=e^{13}+e^{24}$, $%
\sigma_3=e^{14}-e^{23}$ are the three anti-self-dual forms on $\mathbb{R}^4$
and
\begin{equation}  \label{matrixA}
A=\left(\!\!\!
\begin{array}{ccc}
a_{11} & a_{12} & a_{13} \\
a_{21} & a_{22} & a_{23} \\
a_{31} & a_{32} & a_{33}%
\end{array}
\!\right).
\end{equation}
We will denote the norm of $A$ by $|A|$, $|A|^2=\sum_{i,j=1}^3 a_{ij}^2$.

Since $\mathfrak{K}_A$ is isomorphic to $\mathfrak{g(\mathbb{H})}$, if $K_A$
is connected and simply connected it is isomorphic to $G(\mathbb{H})$.
Furthermore, any lattice $\Gamma_A$ gives rise to a (compact) nilmanifold $%
M_A=K_A/\Gamma_A$, which is a $\mathbb{T}^3$-bundle over a $\mathbb{T}^4$
with connection 1-forms of anti-self-dual curvature on the four torus.

Following \cite{FIUVdim7-8} we consider the $G_2$ structure on the Lie group
$K_A$ defined by the 3-form
\begin{equation}  \label{g2-general}
\Theta=\omega_1\wedge e^7+\omega_2\wedge e^5-\omega_3\wedge e^6 + e^{567},
\end{equation}
where
\begin{equation*}
\omega_1=e^{12}+e^{34},\quad \omega_2=e^{13}-e^{24},\quad
\omega_3=e^{14}+e^{23}
\end{equation*}
are the three closed self-dual 2-forms on $\mathbb{R}^4$. The corresponding
Hodge dual 4-form $*\Theta$ is given by
\begin{equation}  \label{g21-general}
*\Theta=\omega_1\wedge e^{56}+\omega_2\wedge e^{67}+\omega_3\wedge
e^{57}+\frac12\omega_1\wedge\omega_1.
\end{equation}
It is easy to check using \eqref{ecus-general} and the property $%
\sigma_i\wedge\omega_j=0$ for $1\leq i,j \leq 3$ that
\begin{equation}  \label{f1-general}
d*\Theta=0, \quad d\Theta\wedge\Theta=0,
\end{equation}
i.e. $\Theta$ is co-calibrated of pure type. According to \cite{FI1,FI2}
this $G_2$ structure solves the gravitino and dilatino equations with
constant dilaton.

Let $f$ be a smooth function on $\mathbb{R}^4$. Following \cite{FIUVdim7-8}
we consider the $G_2$ form given by
\begin{equation}  \label{g2f-general}
\bar\Theta=e^{2f}\Big[\omega_1\wedge e^7+\omega_2\wedge e^5-\omega_3\wedge
e^6\Big]+ e^{567}.
\end{equation}
The corresponding metric $\bar{g}$ on $K_{A}$ has an orthonormal basis of
1-forms given by
\begin{equation}  \label{conf-general}
\bar{e}^{1}=e^{f}\,e^{1},\quad \bar{e}^{2}=e^{f}\,e^{2},\quad \bar{e}%
^{3}=e^{f}\,e^{3},\quad \bar{e}^{4}=e^{f}\,e^{4},\quad \bar{e}%
^{5}=e^{5},\quad \bar{e}^{6}=e^{6},\quad \bar{e}^{7}=e^{7}
\end{equation}
and self-dual form $\bar\omega_i$ and anti-self-dual forms $\bar\sigma_i$
given by
\begin{equation}  \label{conf 2forms}
\bar\omega_i=e^{2f}\omega_i, \qquad \bar\sigma_i=e^{2f}\sigma_i, \quad
i=1,2,3.
\end{equation}
The corresponding Hodge dual 4-form $\bar *\bar\Theta$ is
\begin{equation}  \label{g21f-general}
\bar *\bar\Theta=e^{2f}\Big[\omega_1\wedge e^{56}+\omega_2\wedge e^{67}
+\omega_3\wedge e^{57}+\frac{e^{2f}}2\omega_1\wedge\omega_1\Big].
\end{equation}
It was shown in \cite[Theorem 6.1]{FIUVdim7-8} using \eqref{f1-general} that
\begin{equation}  \label{f1f-general}
d\bar *\bar\Theta=2df\wedge\bar *\bar\Theta, \quad
d\bar\Theta\wedge\bar\Theta=0.
\end{equation}
Then the Lie form $\bar\theta$ is given by
\begin{equation}  \label{lif-general}
\bar\theta= 2df
\end{equation}
and the $G_2$ structure $\bar\Theta$ solves the gravitino and dilatino
equations with non-constant dilaton $\phi=-2f$ \cite{FI1,FI2}.

According to \cite{FI1,FI2}, the torsion of the (+)-connection $\nabla^+$ is
the 3-form
\begin{equation}  \label{torg-general}
T=-*d\Theta + *(\theta\wedge\Theta).
\end{equation}
We calculate from \eqref{ecus-general} and \eqref{g2f-general} that
\begin{equation}  \label{dg2f-general}
d\bar\Theta=2df\wedge\bar\Theta-2df\wedge e^{567}+d e^{567}.
\end{equation}
A substitution of \eqref{dg2f-general} in \eqref{torg-general}, and using %
\eqref{lif-general}, gives
\begin{multline}  \label{torg1-general}
\bar T=\bar *(2df\wedge e^{567}-d e^{567}) =e^{-f}\Big[-2f_1\,\bar{e}%
^{234}+2f_2\,\bar{e}^{134}-2f_3\,\bar{e}^{124}+2f_4\,\bar{e}^{123}\Big] \\
+e^{-2f}\Big[(a_{11}\,\bar{\sigma_1}+a_{12}\,\bar{\sigma_2}+a_{13}\,\bar{%
\sigma_3})\wedge\bar{e}^{5} +(a_{21}\,\bar{\sigma_1}+a_{22}\,\bar{\sigma_2}%
+a_{23}\,\bar{\sigma_3})\wedge\bar{e}^{6} +(a_{31}\,\bar{\sigma_1}+a_{32}\,%
\bar{\sigma_2}+a_{33}\,\bar{\sigma_3})\wedge\bar{e}^{7} \Big],
\end{multline}
where $f_{i}=\frac{\partial f}{\partial x_{i}}$, $1\leq i\leq 4$, and $\bar{%
\sigma_1}=\bar{e}^{12}-\bar{e}^{34}$, $\bar{\sigma_2}=\bar{e}^{13}+\bar{e}%
^{24}$ and $\bar{\sigma_3}=\bar{e}^{14}-\bar{e}^{23}$. Letting $f_{ij}=\frac{%
\partial ^{2}f}{\partial x_{j}\partial x_{i}}$, $1\leq i,j\leq 4$, a short
calculation gives 
\begin{equation}  \label{torsion-general}
d\bar{T} =-e^{-4f}\left[ \triangle e^{2f}+2|A|^2\right] \,\bar{e}^{1234} =-%
\left[ \triangle e^{2f}+2|A|^2 \right] e^{1234},
\end{equation}
where $\triangle
e^{2f}=(e^{2f})_{11}+(e^{2f})_{22}+(e^{2f})_{33}+(e^{2f})_{44}$
 is the
standard Laplacian on $\mathbb{R}^4$.

\subsection{The first Pontrjagin form of the $(-)$-connection}

\label{pon7-general}

From  Koszul's formula, we have that the Levi-Civita connection 1-forms $%
(\omega ^{\bar{g}})_{\bar{j}}^{\bar{\imath}}$ of the metric $\bar{g}$ are
given by
\begin{equation}\label{lc-general}
\begin{array}{ll}
(\omega ^{\bar{g}})_{\bar{j}}^{\bar{\imath}}(\bar{e}_{k}) \!&\! =-\frac{1}{2}\Big(%
\bar{g}(\bar{e}_{i},[\bar{e}_{j},\bar{e}_{k}])-\bar{g}(\bar{e}_{k},[\bar{e}%
_{i},\bar{e}_{j}])+\bar{g}(\bar{e}_{j},[\bar{e}_{k},\bar{e}_{i}])\Big) \\[8pt]
\!&\! =\frac{1}{2}\Big(d\bar{e}^{i}(\bar{e}_{j},\bar{e}_{k})-d\bar{e}^{k}(%
\bar{e}_{i},\bar{e}_{j})+d\bar{e}^{j}(\bar{e}_{k},\bar{e}_{i})\Big)
\end{array}
\end{equation}
taking into account $\bar{g}(\bar{e}_{i},[\bar{e}_{j},\bar{e}_{k}])=-d\bar{e}%
^{i}(\bar{e}_{j},\bar{e}_{k})$. With the help of \eqref{lc-general} we
compute the expressions for the connection 1-forms $(\omega ^{-})_{\bar{j}}^{%
\bar{\imath}}$ of the connection $\nabla ^{-}$,
\begin{equation}
(\omega ^{-})_{\bar{j}}^{\bar{\imath}}=(\omega ^{\bar{g}})_{\bar{j}}^{\bar{%
\imath}}-\frac{1}{2}(\bar{T})_{\bar{j}}^{\bar{\imath}},\qquad \text{ where }%
\qquad (\bar{T})_{\bar{j}}^{\bar{\imath}}(\bar{e}_{k})=\bar{T}(\bar{e}_{i},%
\bar{e}_{j},\bar{e}_{k}).  \label{minus-general}
\end{equation}%
Now, \eqref{minus-general}, \eqref{lc-general} and \eqref{torg1-general}
show that the possibly non-zero connection 1-forms $(\omega ^{-})_{\bar{j}}^{%
\bar{\imath}}$ are given in terms of the basis $\{\bar{e}^{1},\ldots ,\bar{e}%
^{7}\}$ by:

\begin{equation}
\begin{array}{l}
(\omega ^{-})_{\bar{2}}^{\bar{1}}=(\omega ^{-})_{\bar{4}}^{\bar{3}%
}=e^{-f}\left( f_{2}\,\bar{e}^{1}-f_{1}\,\bar{e}^{2}+f_{4}\,\bar{e}%
^{3}-f_{3}\,\bar{e}^{4}\right) ,\;\;\; \\[6pt]
(\omega ^{-})_{\bar{3}}^{\bar{1}}=-(\omega ^{-})_{\bar{4}}^{\bar{2}%
}=e^{-f}\left( f_{3}\,\bar{e}^{1}-f_{4}\,\bar{e}^{2}-f_{1}\,\bar{e}%
^{3}+f_{2}\,\bar{e}^{4}\right) , \\
(\omega ^{-})_{\bar{4}}^{\bar{1}}=(\omega ^{-})_{\bar{3}}^{\bar{2}%
}=e^{-f}\left( f_{4}\,\bar{e}^{1}+f_{3}\,\bar{e}^{2}-f_{2}\,\bar{e}%
^{3}-f_{1}\,\bar{e}^{4}\right) ,\;\;\; \\[6pt]
(\omega ^{-})_{\bar{5}}^{\bar{1}}=e^{-2f}\left( -a_{11}\,\bar{e}^{2}-a_{12}\,%
\bar{e}^{3}-a_{13}\,\bar{e}^{4}\right) ,\;\;\;(\omega ^{-})_{\bar{6}}^{\bar{1%
}}=e^{-2f}\left( -a_{21}\,\bar{e}^{2}-a_{22}\,\bar{e}^{3}-a_{23}\,\bar{e}%
^{4}\right) ,\;\;\; \\[6pt]
(\omega ^{-})_{\bar{7}}^{\bar{1}}=e^{-2f}\left( -a_{31}\,\bar{e}^{2}-a_{32}\,%
\bar{e}^{3}-a_{33}\,\bar{e}^{4}\right) ,\;\;(\omega ^{-})_{\bar{5}}^{\bar{2}%
}=e^{-2f}\left( a_{11}\,\bar{e}^{1}+a_{13}\,\bar{e}^{3}-a_{12}\,\bar{e}%
^{4}\right) ,\;\;\; \\[6pt]
(\omega ^{-})_{\bar{6}}^{\bar{2}}=e^{-2f}\left( a_{21}\,\bar{e}^{1}+a_{23}\,%
\bar{e}^{3}-a_{22}\,\bar{e}^{4}\right) ,\;\;\;(\omega ^{-})_{\bar{7}}^{\bar{2%
}}=e^{-2f}\left( a_{31}\,\bar{e}^{1}+a_{33}\,\bar{e}^{3}-a_{32}\,\bar{e}%
^{4}\right) ,\;\;\; \\
(\omega ^{-})_{\bar{5}}^{\bar{3}}=e^{-2f}\left( a_{12}\,\bar{e}^{1}-a_{13}\,%
\bar{e}^{2}+a_{11}\,\bar{e}^{4}\right) ,\;\;\;(\omega ^{-})_{\bar{6}}^{\bar{3%
}}=e^{-2f}\left( a_{22}\,\bar{e}^{1}-a_{23}\,\bar{e}^{2}+a_{21}\,\bar{e}%
^{4}\right) ,\;\;\; \\[6pt]
(\omega ^{-})_{\bar{7}}^{\bar{3}}=e^{-2f}\left( a_{32}\,\bar{e}^{1}-a_{33}\,%
\bar{e}^{2}+a_{31}\,\bar{e}^{4}\right) ,\;\;\;(\omega ^{-})_{\bar{5}}^{\bar{4%
}}=e^{-2f}\left( a_{13}\,\bar{e}^{1}+a_{12}\,\bar{e}^{2}-a_{11}\,\bar{e}%
^{3}\right) ,\;\;\; \\
(\omega ^{-})_{\bar{6}}^{\bar{4}}=e^{-2f}\left( a_{23}\,\bar{e}^{1}+a_{22}\,%
\bar{e}^{2}-a_{21}\,\bar{e}^{3}\right) ,\;\;\;(\omega ^{-})_{\bar{7}}^{\bar{4%
}}=e^{-2f}\left( a_{33}\,\bar{e}^{1}+a_{32}\,\bar{e}^{2}-a_{31}\,\bar{e}%
^{3}\right) .%
\end{array}
\label{connection-forms-general}
\end{equation}

A long straightforward calculation using \eqref{connection-forms-general}
gives in terms of the basis $\{ \bar{e}^{1},\ldots,\bar{e}^{7} \}$ the
following formulas for the curvature 2-forms $(\Omega^{-})^{\bar i}_{\bar j}$
of the connection $\nabla^{-}$:
\begin{equation*}
\begin{array}[t]{rl}
(\Omega^{-})^{\bar 1}_{\bar 2}=\!\! & \!\!
-e^{-2f}[f_{11}+f_{22}+2f_{3}^{2}+2f_{4}^{2}+(a_{11}^{2}+a_{21}^{2}+a_{31}^{2})e^{-2f}] \,%
\bar{e}^{12} \\[4pt]
\!\! &
+e^{-2f}[f_{14}-f_{23}-2f_{1}f_{4}+2f_{2}f_{3}-(a_{11}a_{12}+a_{21}a_{22}+a_{31}a_{32})e^{-2f}] \,%
\bar{\sigma}_{2} \\[4pt]
\!\! & -e^{-2f}[f_{13}+f_{24}-2f_{1}f_{3}-2f_{2}f_{4}+(a_{11}a_{13}+
a_{21}a_{23}+ a_{31} a_{33})e^{-2f}]\, \bar{\sigma}_{3} \\[4pt]
\!\! & -e^{-2f}[f_{33}+f_{44}+2f_{1}^{2}+2f_{2}^{2}+(a_{12}^2+ a_{22}^2+
a_{32}^2+a_{13}^2+a_{23}^2+ a_{33}^2)e^{-2f}]\, \bar{e}^{34}, \\[6pt]
(\Omega^{-})^{\bar 1}_{\bar 3}=\!\! & \!\!
-e^{-2f}[f_{14}+f_{23}-2f_{1}f_{4}-2f_{2}f_{3}+(a_{11} a_{12}+ a_{21}
a_{22}+ a_{31} a_{32})e^{-2f}] \,\bar{\sigma}_{1} \\[4pt]
\!\! & -e^{-2f}[f_{11}+f_{33}+2f_{2}^{2}+2f_{4}^{2}+(a_{12}^2+a_{22}^2+
a_{32}^2)e^{-2f}] \,\bar{e}^{13} \\[4pt]
\!\! & +e^{-2f}[f_{12}-f_{34}-2f_{1}f_{2}+2f_{3}f_{4}+(a_{12} a_{13}+ a_{22}
a_{23}+ a_{32} a_{33})e^{-2f}] \,\bar{\sigma}_{3} \\[4pt]
\!\! & +e^{-2f}[f_{22}+f_{44}+2f_{1}^{2}+2f_{3}^{2}+(a_{11}^2+
a_{21}^2+a_{31}^2+ a_{13}^2+a_{23}^2+ a_{33}^2)e^{-2f}] \,\bar{e}^{24}, \\%
[6pt]
(\Omega^{-})^{\bar 1}_{\bar 4}=\!\! & \!\!
e^{-2f}[f_{13}-f_{24}-2f_{1}f_{3}+2f_{2}f_{4} -( a_{11} a_{13}+ a_{21}
a_{23}+ a_{31} a_{33} )e^{-2f}] \,\bar{\sigma}_{1} \\[4pt]
\!\! & \!\! -e^{-2f}[f_{12}+f_{34}-2f_{1}f_{2}-2f_{3}f_{4} +( a_{12} a_{13}+
a_{22} a_{23}+ a_{32} a_{33} )e^{-2f} ] \,\bar{\sigma}_{2} \\[4pt]
\!\! & \!\! -e^{-2f}[f_{11}+f_{44}+2f_{2}^{2}+2f_{3}^{2}+( a_{13}^2+
a_{23}^2+ a_{33}^2 )e^{-2f}] \,\bar{e}^{14} \\[4pt]
\!\! & \!\! -e^{-2f}[f_{22}+f_{33}+2f_{1}^{2}+2f_{4}^{2}+(a_{11}^2+
a_{21}^2+ a_{31}^2+ a_{12}^2+ a_{22}^2+ a_{32}^2)e^{-2f}] \,\bar{e}^{23},%
\end{array}%
\end{equation*}
\begin{equation*}
\begin{array}[t]{rl}
(\Omega^{-})^{\bar 1}_{\bar 5}=\!\! & \!\! 2 e^{-3f} \left[ ( a_{11} f_{1}-
a_{13} f_{3}+ a_{12} f_{4} ) \,\bar{\sigma}_{1} +( a_{12} f_{1}+ a_{13}
f_{2}- a_{11} f_{4} ) \,\bar{\sigma}_{2} +( a_{13} f_{1}- a_{12} f_{2}+
a_{11} f_{3} ) \,\bar{\sigma}_{3} \right], \\[6pt]
(\Omega^{-})^{\bar 1}_{\bar 6}=\!\! & \!\! 2 e^{-3f} \left[ ( a_{21} f_{1}-
a_{23} f_{3}+ a_{22} f_{4} ) \,\bar{\sigma}_{1} +( a_{22} f_{1}+ a_{23}
f_{2}- a_{21} f_{4} ) \,\bar{\sigma}_{2} +( a_{23} f_{1}- a_{22} f_{2}+
a_{21} f_{3} ) \,\bar{\sigma}_{3} \right], \\[6pt]
(\Omega^{-})^{\bar 1}_{\bar 7}=\!\! & \!\! 2 e^{-3f} \left[ ( a_{31} f_{1}-
a_{33} f_{3}+ a_{32} f_{4} ) \,\bar{\sigma}_{1} +( a_{32} f_{1}+ a_{33}
f_{2}- a_{31} f_{4} ) \,\bar{\sigma}_{2} +( a_{33} f_{1}- a_{32} f_{2}+
a_{31} f_{3} ) \,\bar{\sigma}_{3} \right],%
\end{array}%
\end{equation*}
\begin{equation*}
\begin{array}[t]{rl}
(\Omega^{-})^{\bar 2}_{\bar 3}=\!\! & \!\!
e^{-2f}[f_{13}-f_{24}-2f_{1}f_{3}+2f_{2}f_{4} +( a_{11} a_{13}+ a_{21}
a_{23}+ a_{31} a_{33} )e^{-2f}] \,\bar{\sigma}_{1} \\[4pt]
\!\! & \!\! -e^{-2f}[f_{12}+f_{34}-2f_{1}f_{2}-2f_{3}f_{4} -( a_{12} a_{13}+
a_{22} a_{23}+ a_{32} a_{33} )e^{-2f}] \,\bar{\sigma}_{2} \\[4pt]
\!\! & \!\! -e^{-2f}[f_{11}+f_{44}+2f_{2}^{2}+2f_{3}^{2} +( a_{11}^2+
a_{21}^2+ a_{31}^2+ a_{12}^2+ a_{22}^2+ a_{32}^2 )e^{-2f}] \,\bar{e}^{14} \\%
[4pt]
\!\! & \!\! -e^{-2f}[f_{22}+f_{33}+2f_{1}^{2}+2f_{4}^{2} +( a_{13}^2+
a_{23}^2+ a_{33}^2 )e^{-2f}] \,\bar{e}^{23}, \\[6pt]
(\Omega^{-})^{\bar 2}_{\bar 4}=\!\! & \!\!
e^{-2f}[f_{14}+f_{23}-2f_{1}f_{4}-2f_{2}f_{3} -( a_{11} a_{12}+ a_{21}
a_{22}+ a_{31} a_{32} )e^{-2f}] \,\bar{\sigma}_{1} \\[4pt]
\!\! & \!\! +e^{-2f}[f_{11}+f_{33}+2f_{2}^{2}+2f_{4}^{2} +( a_{11}^2+
a_{21}^2+ a_{31}^2+ a_{13}^2+ a_{23}^2+ a_{33}^2 )e^{-2f}] \,\bar{e}^{13} \\%
[4pt]
\!\! & \!\! -e^{-2f}[f_{12}-f_{34}-2f_{1}f_{2}+2f_{3}f_{4} +( a_{12} a_{13}+
a_{22} a_{23}+ a_{32} a_{33} )e^{-2f}] \,\bar{\sigma}_{3} \\[4pt]
\!\! & \!\! -e^{-2f}[f_{22}+f_{44}+2f_{1}^{2}+2f_{3}^{2} +( a_{12}^2+
a_{22}^2+ a_{32}^2 )e^{-2f}] \,\bar{e}^{24},%
\end{array}%
\end{equation*}
\begin{equation*}
\begin{array}[t]{rl}
(\Omega^{-})^{\bar 2}_{\bar 5}=\!\! & \!\! 2 e^{-3f} \left[ ( a_{11} f_{2}-
a_{12} f_{3}- a_{13} f_{4} ) \,\bar{\sigma}_{1} -( a_{13} f_{1}- a_{12}
f_{2}- a_{11} f_{3} ) \,\bar{\sigma}_{2} +( a_{12} f_{1}+ a_{13} f_{2}+
a_{11} f_{4} ) \,\bar{\sigma}_{3} \right], \\[6pt]
(\Omega^{-})^{\bar 2}_{\bar 6}=\!\! & \!\! 2 e^{-3f} \left[ ( a_{21} f_{2}-
a_{22} f_{3}- a_{23} f_{4} ) \,\bar{\sigma}_{1} -( a_{23} f_{1}- a_{22}
f_{2}- a_{21} f_{3} ) \,\bar{\sigma}_{2} +( a_{22} f_{1}+ a_{23} f_{2}+
a_{21} f_{4} ) \,\bar{\sigma}_{3} \right], \\[6pt]
(\Omega^{-})^{\bar 2}_{\bar 7}=\!\! & \!\! 2 e^{-3f} \left[ ( a_{31} f_{2}-
a_{32} f_{3}- a_{33} f_{4} ) \,\bar{\sigma}_{1} -( a_{33} f_{1}- a_{32}
f_{2}- a_{31} f_{3} ) \,\bar{\sigma}_{2} +( a_{32} f_{1}+ a_{33} f_{2}+
a_{31} f_{4} ) \,\bar{\sigma}_{3} \right],%
\end{array}%
\end{equation*}
\begin{equation*}
\begin{array}[t]{rl}
(\Omega^{-})^{\bar 3}_{\bar 4}=\!\! & \!\!
-e^{-2f}[f_{11}+f_{22}+2f_{3}^{2}+2f_{4}^{2} +( a_{12}^2+ a_{22}^2+
a_{32}^2+ a_{13}^2+ a_{23}^2+ a_{33}^2 )e^{-2f}) \bar{e}^{12} \\[4pt]
\!\! & \!\! +e^{-2f}[f_{14}-f_{23}-2f_{1}f_{4}+2f_{2}f_{3} +( a_{11} a_{12}+
a_{21} a_{22}+ a_{31} a_{32} )e^{-2f}] \,\bar{\sigma}_{2} \\[4pt]
\!\! & \!\!-e^{-2f}[f_{13}+f_{24}-2f_{1}f_{3}-2f_{2}f_{4} -( a_{11} a_{13}+
a_{21} a_{23}+ a_{31} a_{33} )e^{-2f}] \,\bar{\sigma}_{3} \\[4pt]
\!\! & \!\! -e^{-2f}[f_{33}+f_{44}+2f_{1}^{2}+2f_{2}^{2} +( a_{11}^2+
a_{21}^2+ a_{31}^2 )e^{-2f}) \bar{e}^{34},%
\end{array}%
\end{equation*}
\begin{equation*}
\begin{array}[t]{rl}
(\Omega^{-})^{\bar 3}_{\bar 5}=\!\! & \!\! 2 e^{-3f} \left[ ( a_{13} f_{1}+
a_{12} f_{2}+ a_{11} f_{3} ) \,\bar{\sigma}_{1} -( a_{11} f_{2}- a_{12}
f_{3}+ a_{13} f_{4} ) \,\bar{\sigma}_{2} -( a_{11} f_{1}- a_{13} f_{3}-
a_{12} f_{4} ) \,\bar{\sigma}_{3} \right], \\[6pt]
(\Omega^{-})^{\bar 3}_{\bar 6}=\!\! & \!\! 2 e^{-3f} \left[ ( a_{23} f_{1}+
a_{22} f_{2}+ a_{21} f_{3} ) \,\bar{\sigma}_{1} -( a_{21} f_{2}- a_{22}
f_{3}+ a_{23} f_{4} ) \,\bar{\sigma}_{2} -( a_{21} f_{1}- a_{23} f_{3}-
a_{22} f_{4} ) \,\bar{\sigma}_{3} \right], \\[6pt]
(\Omega^{-})^{\bar 3}_{\bar 7}=\!\! & \!\! 2 e^{-3f} \left[ ( a_{33} f_{1}+
a_{32} f_{2}+ a_{31} f_{3} ) \,\bar{\sigma}_{1} -( a_{31} f_{2}- a_{32}
f_{3}+ a_{33} f_{4} ) \,\bar{\sigma}_{2} -( a_{31} f_{1}- a_{33} f_{3}-
a_{32} f_{4} ) \,\bar{\sigma}_{3} \right],%
\end{array}%
\end{equation*}
\begin{equation*}
\begin{array}[t]{rl}
(\Omega^{-})^{\bar 4}_{\bar 5}=\!\! & \!\! 2 e^{-3f} \left[ -( a_{12} f_{1}-
a_{13} f_{2}- a_{11} f_{4} ) \,\bar{\sigma}_{1} +( a_{11} f_{1}+ a_{13}
f_{3}+ a_{12} f_{4} ) \,\bar{\sigma}_{2} -( a_{11} f_{2}+ a_{12} f_{3}-
a_{13} f_{4} ) \,\bar{\sigma}_{3} \right], \\[6pt]
(\Omega^{-})^{\bar 4}_{\bar 6}=\!\! & \!\! 2 e^{-3f} \left[ -( a_{22} f_{1}-
a_{23} f_{2}- a_{21} f_{4} ) \,\bar{\sigma}_{1} +( a_{21} f_{1}+ a_{23}
f_{3}+ a_{22} f_{4} ) \,\bar{\sigma}_{2} -( a_{21} f_{2}+ a_{22} f_{3}-
a_{23} f_{4} ) \,\bar{\sigma}_{3} \right], \\[6pt]
(\Omega^{-})^{\bar 4}_{\bar 7}=\!\! & \!\! 2 e^{-3f} \left[ -( a_{32} f_{1}-
a_{33} f_{2}- a_{31} f_{4} ) \,\bar{\sigma}_{1} +( a_{31} f_{1}+ a_{33}
f_{3}+ a_{32} f_{4} ) \,\bar{\sigma}_{2} -( a_{31} f_{2}+ a_{32} f_{3}-
a_{33} f_{4} ) \,\bar{\sigma}_{3} \right],%
\end{array}%
\end{equation*}
\begin{equation*}
\begin{array}[t]{rl}
(\Omega^{-})^{\bar 5}_{\bar 6}=\!\! & \!\! 2 e^{-4f} \left[ ( a_{12}
a_{23}-a_{13} a_{22} ) \,\bar{\sigma}_{1} -( a_{11} a_{23}-a_{13} a_{21} ) \,%
\bar{\sigma}_{2} +( a_{11} a_{22}-a_{12} a_{21} ) \,\bar{\sigma}_{3} \right],
\\[6pt]
(\Omega^{-})^{\bar 5}_{\bar 7}=\!\! & \!\! 2 e^{-4f} \left[ ( a_{12}
a_{33}-a_{13} a_{32} ) \,\bar{\sigma}_{1} -( a_{11} a_{33}-a_{13} a_{31} ) \,%
\bar{\sigma}_{2} +( a_{11} a_{32}-a_{12} a_{31} ) \,\bar{\sigma}_{3} \right],
\\[6pt]
(\Omega^{-})^{\bar 6}_{\bar 7}=\!\! & \!\! 2 e^{-4f} \left[ ( a_{22}
a_{33}-a_{23} a_{32} ) \,\bar{\sigma}_{1} -( a_{21} a_{33}-a_{23} a_{31} ) \,%
\bar{\sigma}_{2} +( a_{21} a_{32}-a_{22} a_{31} ) \,\bar{\sigma}_{3} \right].%
\end{array}%
\end{equation*}

A long calculation based on the formulas for the curvature 2-form $%
(\Omega^{-})^{\bar i}_{\bar j}$ of $\nabla^-$ gives

\begin{prop}
The first Pontrjagin form of $\nabla^{-}$ is a scalar multiple of $e^{1234}$
given by
\begin{equation}  \label{p1-general}
\pi ^{2}p_{1}(\nabla ^{-}) =\left[ \mathcal{F}_2[f]+\triangle_4 f -\frac{3}{8%
} |A|^{2} \triangle e^{-2f} \right] {e}^{1234},
\end{equation}
where $\mathcal{F}_2[f]$ is the 2-Hessian of $f$, i.e., the sum of all
principle $2\times 2$-minors of the Hessian, and $\triangle_4 f=div(|\nabla
f|^2\nabla f)$ is the 4-Laplacian of $f$.
\end{prop}

The above Proposition shows, in particular, that even though the curvature
2-forms of $\nabla^-$ are quadratic in the gradient of the dilaton, the
Pontrjagin form of $\nabla^-$ is also quadratic in these terms. Furthermore,
if $f$ depends on two of the variables then $\mathcal{F}_2[f]=det (Hess f)$
while if $f$ is a function of one variable $\mathcal{F}_2[f]$ vanishes.

\section{A conformally compact solution with negative $\protect\alpha%
^{\prime}$}

\label{s:compact ex}

In this section we give our first main result. Recall that $K_A$ is the
connected simply connected Lie group with Lie algebra $\mathfrak{K}_A$
determined by \eqref{ecus-general}. Due to the results recalled in Section \ref{ss:dim7 constr} the remaining part is to solve the
anomaly cancellation condition. This we will achieve for the $G_2$
structure \eqref{g2f-general} with the torsion term \eqref{torsion-general},
the Pontrjagin form \eqref{p1-general} of the $\nabla^{-}$ connection, and
the $G_2$-instanton defined below.


\begin{prop}
\label{instanton2-general} Let $\mathrm{D}_{\Lambda }$, $\Lambda=(%
\lambda_{ij}) \in {\mathfrak{g}\mathfrak{l}}_3(\mathbb{R})$, be the linear connection on the
Lie group $K_{A}$ whose possibly non-zero 1-forms are given as follows
\begin{equation*}
\begin{array}{l}
(\omega^{\mathrm{D}_{\Lambda }})_{\bar 2}^{\bar 1} =-(\omega^{\mathrm{D}%
_{\Lambda}})_{\bar 1}^{\bar 2} =-(\omega^{\mathrm{D}_{\Lambda }})_{\bar
4}^{\bar 3} =(\omega^{\mathrm{D}_{\Lambda }})_{\bar 3}^{\bar 4}
=\lambda_{11}\,\bar{e}^{5}+\lambda_{12}\,\bar{e}^{6}+\lambda_{13}\,\bar{e}%
^{7}, \\[8pt]
(\omega^{\mathrm{D}_{\Lambda }})_{\bar 3}^{\bar 1} =-(\omega^{\mathrm{D}%
_{\Lambda}})_{\bar 1}^{\bar 3} =(\omega^{\mathrm{D}_{\Lambda }})_{\bar
4}^{\bar 2} =-(\omega^{\mathrm{D}_{\Lambda }})_{\bar 2}^{\bar 4}
=\lambda_{21}\,\bar{e}^{5}+\lambda_{22}\,\bar{e}^{6}+\lambda_{23}\,\bar{e}%
^{7}, \\[8pt]
(\omega^{\mathrm{D}_{\Lambda }})_{\bar 4}^{\bar 1} =-(\omega^{\mathrm{D}%
_{\Lambda}})_{\bar 1}^{\bar 4} =-(\omega^{\mathrm{D}_{\Lambda }})_{\bar
3}^{\bar 2} =(\omega^{\mathrm{D}_{\Lambda }})_{\bar 2}^{\bar 3}
=\lambda_{31}\,\bar{e}^{5}+\lambda_{32}\,\bar{e}^{6}+\lambda_{33}\,\bar{e}%
^{7}.%
\end{array}%
\end{equation*}
Then, $\mathrm{D}_{\Lambda}$ is a $G_2$-instanton with respect to the $G_2$
structure defined by \eqref{g2f-general} which preserves the metric if and
only if $\mathrm{rank}(\Lambda) \leq 1$.
\end{prop}

\begin{proof}
Let us use the notation $\Lambda _{ijkl}=\lambda _{ik}\lambda _{jl}-\lambda
_{jk}\lambda _{il}=\mathrm{det}\,\left( \!\!\!%
\begin{array}{cc}
\lambda _{ik} & \lambda _{il} \\
\lambda _{jk} & \lambda _{jl}%
\end{array}%
\!\right) $ for the $2\times 2$ minors of $\Lambda$. A direct calculation using %
\eqref{ecus-general} shows that the possibly non-zero curvature forms $%
(\Omega ^{\mathrm{D}_{\Lambda }})_{\bar{j}}^{\bar{\imath}}$ of the
connection $\mathrm{D}_{\Lambda }$ are:
\begin{equation*}
\begin{array}{ll}
(\Omega ^{\mathrm{D}_{\Lambda }})_{\bar{2}}^{\bar{1}}\! & \!=-(\Omega ^{%
\mathrm{D}_{\Lambda }})_{\bar{1}}^{\bar{2}}=-(\Omega ^{\mathrm{D}_{\Lambda
}})_{\bar{4}}^{\bar{3}}=(\Omega ^{\mathrm{D}_{\Lambda }})_{\bar{3}}^{\bar{4}%
}=\ \ e^{-2f}(a_{11}\lambda _{11}+a_{21}\lambda _{12}+a_{31}\lambda _{13})\,%
\bar{\sigma}_{1} \\[4pt]
& \ \ \ +e^{-2f}(a_{12}\lambda _{11}+a_{22}\lambda _{12}+a_{32}\lambda
_{13})\,\bar{\sigma}_{2}+e^{-2f}(a_{13}\lambda _{11}+a_{23}\lambda
_{12}+a_{33}\lambda _{13})\,\bar{\sigma}_{3} \\[4pt]
& \ \ \ +\ 2\Lambda _{2312}\,\bar{e}^{56}+2\Lambda _{2313}\,\bar{e}%
^{57}+2\Lambda _{2323}\,\bar{e}^{67}, \\[8pt]
(\Omega ^{\mathrm{D}_{\Lambda }})_{\bar{3}}^{\bar{1}}\! & \!=-(\Omega ^{%
\mathrm{D}_{\Lambda }})_{\bar{1}}^{\bar{3}}=(\Omega ^{\mathrm{D}_{\Lambda
}})_{\bar{4}}^{\bar{2}}=-(\Omega ^{\mathrm{D}_{\Lambda }})_{\bar{2}}^{\bar{4}%
}=\ \ e^{-2f}(a_{11}\lambda _{21}+a_{21}\lambda _{22}+a_{31}\lambda _{23})\,%
\bar{\sigma}_{1} \\[4pt]
& \ \ \ +e^{-2f}(a_{12}\lambda _{21}+a_{22}\lambda _{22}+a_{32}\lambda
_{23})\,\bar{\sigma}_{2}+e^{-2f}(a_{13}\lambda _{21}+a_{23}\lambda
_{22}+a_{33}\lambda _{23})\,\bar{\sigma}_{3} \\[6pt]
& \ \ \ -\ 2\Lambda _{1312}\,\bar{e}^{56}-2\Lambda _{1313}\,\bar{e}%
^{57}-2\Lambda _{1323}\,\bar{e}^{67}, \\[8pt]
(\Omega ^{\mathrm{D}_{\Lambda }})_{\bar{4}}^{\bar{1}}\! & \!=-(\Omega ^{%
\mathrm{D}_{\Lambda }})_{\bar{1}}^{\bar{4}}=-(\Omega ^{\mathrm{D}_{\Lambda
}})_{\bar{3}}^{\bar{2}}=(\Omega ^{\mathrm{D}_{\Lambda }})_{\bar{2}}^{\bar{3}%
}=\ \ e^{-2f}(a_{11}\lambda _{31}+a_{21}\lambda _{32}+a_{31}\lambda _{33})\,%
\bar{\sigma}_{1} \\[4pt]
& \ \ \ +e^{-2f}(a_{12}\lambda _{31}+a_{22}\lambda _{32}+a_{32}\lambda
_{33})\,\bar{\sigma}_{2}+e^{-2f}(a_{13}\lambda _{31}+a_{23}\lambda
_{32}+a_{33}\lambda _{33})\,\bar{\sigma}_{3} \\[6pt]
& \ \ \ +\ 2\Lambda _{1212}\,\bar{e}^{56}+2\Lambda _{1213}\,\bar{e}%
^{57}+2\Lambda _{1223}\,\bar{e}^{67}.%
\end{array}%
\end{equation*}%
Now, it is straightforward to see that $\mathrm{D}_{\Lambda }$ satisfies %
\eqref{in2} if and only if all the $2\times 2$ minors $\Lambda _{ijkl}$ of
the matrix $\Lambda $ vanish. Therefore, $\mathrm{D}_{\Lambda }$ is a $G_{2}$%
-instanton if and only $\mathrm{rank}(\Lambda )\leq 1$.
\end{proof}

\begin{cor}
\label{instanton2-cor-general} For $\Lambda=(\lambda_{ij}) \in {\mathfrak{g}\mathfrak{l}}_3(\mathbb{R})$
a matrix of rank one, let $\mathrm{D}_{\Lambda}$ be the $G_2$%
-instanton defined in Proposition~\ref{instanton2-general}. Then, the first
Pontrjagin form $p_{1}(\mathrm{D}_{\Lambda})$ of the $G_2$-instanton $%
\mathrm{D}_{\Lambda}$ is given by
\begin{equation}  \label{abinst-general}
8\pi ^{2}p_{1}(\mathrm{D}_{\Lambda})= -4\lambda^2\,e^{1234},
\end{equation}
where $\lambda=|\Lambda\, A|$ is the norm of the product matrix $\Lambda\, A$.
\end{cor}

\begin{proof}
Since the $2\times 2$ minors $\Lambda_{ijkl}$ are all zero, the formulas for the
curvature forms $(\Omega^{\mathrm{D}_{\Lambda}})_{\bar j}^{\bar i}$ given in
the proof of Proposition~\ref{instanton2-general} imply the claimed
identity.

\end{proof}

We turn to the proof of our first main result.

\begin{thrm}
\label{t:cpmct strominger} The conformally compact manifold $%
M^7=(\Gamma\backslash K_{A},\bar{\Theta},\nabla^-, D_{\Lambda}, f)$ is a $%
G_2 $-manifold which solves the Strominger system with non-constant dilaton $%
f$, non-trivial flux $H=\bar T$, non-flat instanton $D_{\Lambda}$ using the
first Pontrjagin form of $\nabla^-$ and negative $\alpha^{\prime }$. The
dilaton $f$ depends on one variable and is determined as a real slice of the
Weierstrass' elliptic function.

The conformally compact manifold $M^7=(\Gamma\backslash K_{A},\bar{\Theta}%
,\nabla^-, D_{\Lambda}, f)$ satisfies the heterotic equations of motion %
\eqref{mot} up to first order of $\alpha^{\prime }$.
\end{thrm}


\begin{proof}
By the construction in Section \ref{ss:dim7 constr} we are left with solving
the anomaly cancellation condition $d\bar{T}=\frac{\alpha ^{\prime }}{4}8\pi
^{2}\Big(p_{1}(\nabla ^{-})-p_{1}(D_\Lambda)\Big)$, which in our case taking
into account \eqref{torsion-general}, \eqref{p1-general} and %
\eqref{abinst-general} becomes \emph{the single}  non-linear equation
\begin{equation}  \label{e:anomaly negative alpha}
\triangle e^{2f}+2|A|^2 +\frac{\alpha ^{\prime }}{4}\left[ 8\mathcal{F}%
_2[f]+8\triangle_4 f -3 |A|^{2} \triangle e^{-2f} +4\lambda^2\right]=0.
\end{equation}
Up to relabelling the constants, this is the same equation as the one
obtained through the anomaly cancellation that appeared in \cite[Section 4.2]%
{FIUVas}. Accordingly, we assume that the function $f$ depends on one
variable, $f=f(x^{1})$,
and for a \emph{negative} $\alpha ^{\prime }$ we choose $2|A|^2+\alpha
^{\prime }\lambda ^{2}=0$, i.e., we let $\alpha ^{\prime }=-\alpha^2$ so
that $2|A|^2=\alpha ^{2}\lambda ^{2}$. This simplifies
\eqref{e:anomaly
negative alpha} to the ordinary differential equation
\begin{equation}  \label{solv4}
\left( e^{2f}\right) ^{\prime }+\frac34\alpha ^{2}|A|^2\left( e^{-2f}\right)
^{\prime }-2\alpha ^{2 }f^{\prime 3}=C_0=const.
\end{equation}%
A solution of the last equation for $C_0=0$ was found in \cite[Section 4.2]%
{FIUVas}. For ease of reading we repeat the key steps of the derivation in
order to obtain a seven dimensional solution of the Strominger system. The
substitution $u=\alpha^{-2} e^{2f}$ allows us to write \eqref{solv4} in the
form
\begin{equation*}
\left( e^{2f}\right) ^{\prime }+\frac34\alpha ^{2}|A|^2\left( e^{-2f}\right)
^{\prime }-2\alpha ^{2 }f^{\prime 3}=\frac{\alpha^2 u^{\prime }}{4u^{3}}%
\left( 4u^{3}-3\frac {|A|^2}{\alpha ^{2 }}u-u^{\prime 2}\right) .
\end{equation*}%
For $C_0=0$ we shall solve the following ordinary differential equation for
the function $u=u(x^1)>0$ 
\begin{equation}  \label{solv5}
u^{\prime 2}={4}u^{3}-3\frac {|A|^2}{\alpha ^{2 }}u=4u\left( u-d\right)
\left( u+d\right) ,\qquad d=\sqrt{3|A|^2}/ \alpha.
\end{equation}%
Replacing the real derivative with the complex derivative leads to the
Weierstrass' equation
\begin{equation}  \label{e:Weirstrass eqn}
\left (\frac {d\, \mathcal{P}}{dz}\right )^2=4\mathcal{P}\left( \mathcal{P}%
-d\right) \left( \mathcal{P}+d\right)
\end{equation}
for the doubly periodic Weierstrass $\mathcal{P}$ function with a pole at
the origin. As well known, \cite{Erd} and \cite{Ahl}, near the origin $%
\mathcal{P}$ has the expansion
\begin{equation*}
\mathcal{P}(z) = \frac {1}{z^2} + \frac {d^2}{5} z^2 + d_1z^6 + \cdots,
\end{equation*}
which has no $z^4$ term and only even powers of $z$. Furthermore, see \cite%
{Erd} and \cite{Ahl}, letting $\tau_{\pm}$ be the basic half-periods such
that $\tau _{+}$ is real and $\tau _{-}$ is purely imaginary we have that $%
\mathcal{P}$ is real valued on the lines $\mathfrak{R}\mathfrak{e}\,z=m\tau
_{+}$ or $\mathfrak{I}\mathfrak{m}\,z=im\tau _{-}$, $m\in \mathbb{Z}$. In
the fundamental region centered at the origin, where $\mathcal{P}$ has a
pole of order two, we have that $\mathcal{P}(z)$ decreases from $+\infty $
to $a$ to $0$ to $-a$ to $-\infty $ as $z$ varies along the sides of the
half-period rectangle from $0$ to $\tau _{+}$ to $\tau _{+}+\tau _{-}$ to $%
\tau _{-}$ to $0$.

Thus, $u(x^1)=\mathcal{P}(x^1)$ defines a non-negative $2\tau_{+}$-periodic
function with singularities at the points $2n \tau_{+}$, $n\in \mathbb{Z}$,
which solves the real equation \eqref{solv5}. From the Laurent expansion of
the Weierstrass' function it follows
\begin{equation*}
u(x_1)=\frac {1}{(x^1)^2}\left (1+ \frac {d^2}{5} (x^1)^4 + \cdots \right).
\end{equation*}
By construction, $f= \frac 12 \ln (\alpha^2 u)$ is a periodic function with
singularities on the real line which is a solution to equation %
\eqref{e:anomaly negative alpha}. {\ 
Therefore the $G_2$ structure defined by $\bar \Theta$ descends to the $7$%
-dimensional nilmanifold $M^{7}=\Gamma \backslash K_{A}$ with singularity,
determined by the singularity of $u$, where $K_{A}$ is the 2-step nilpotent
Lie group with Lie algebra $\mathfrak{K}_{A}$, defined by %
\eqref{ecus-general}, and $\Gamma $ is a lattice with the same period as $f$%
, i.e., $2 \tau_{+}$ in all variables. In fact, as seen from the asymptotic
behavior of $u$, $M^7$ is the total space of a $\mathbb{T}^3$ bundle over
the asymptotically hyperbolic manifold $M^4$ with metric
\begin{equation*}
\bar g_H =u(x^1)\left ( (dx^1)^2+(dx^2)^2 + (dx^3)^2 +(dx^4)^2 \right ),
\end{equation*}
which is a conformally compact 4-torus with conformal boundary at infinity a
flat 3-torus. Thus, we 
conclude that there is a complete solution with non-constant dilaton,
non-trivial instanton and flux and with a negative $\alpha ^{\prime }$
parameter. }

The last statement follows from the fact that the $(-)$-connection is an
instanton up to the first order of $\alpha^{\prime }$. This completes the
proof of Theorem \ref{t:cpmct strominger}.
\end{proof}

From the apparent $\mathbb{Z}_2 $-symmetry of $u$ determined by the symmetry
with respect to the line $x^1=\tau_+ $ we also obtain a solution on the
quotient $M^7/\mathbb{Z}_2$.

\section{A complete solution with positive $\protect\alpha^{\prime}$}

\label{s:non-compact ex} In this section we exhibit a solution of the
Strominger system using again the $G_2$ structure \eqref{g2f-general} by
solving the anomaly cancellation condition with torsion term %
\eqref{torsion-general}, the Pontrjagin form \eqref{p1-general} of the $%
\nabla^{-}$ connection, and the $G_2$-instanton defined with the help of
Lemma \ref{nablamin-inst-general}.

As usual, the $(\pm)$-connections of the $G_2$ structure $\bar{\Theta}$
are defined by the formula $\nabla ^{\pm}=\nabla ^{\bar{g}}\pm\frac12\bar{T}$%
, where $\nabla ^{\bar{g}}$ is the Levi-Civita connection of the metric $%
\bar g$ and the torsion is determined in \eqref{torg1-general}. The
curvature of the connections $\nabla^{\pm}$ are denoted by $R^{\pm}$.

\begin{lemma} \label{nablamin-inst-general}
The $(-)$-connection of the $G_2$ structure
$\bar{\Theta}$ is a $G_2$ instanton with respect to $\bar{\Theta}$ if
and only if the torsion 3-form is closed, $d\bar T=0$, i.e. the dilaton
function $f$ satisfies the equality
\begin{equation}  \label{in11-general}
\triangle e^{2f}+2|A|^2=0.
\end{equation}
\end{lemma}

\begin{proof}
Let $\{\bar{e}_{1},\ldots,\bar{e}_{7}\}$ be the orthonormal basis dual to
$\{\bar{e}^{1},\ldots,\bar{e}^{7}\}$. Using \eqref{dtr} we investigate the $%
G_2$ instanton condition \eqref{in2} for $R^-$ as follows
\begin{equation} \label{inbar}
\begin{array}{ll}
0 \!&\! =\sum_{i,j=1}^7R^-(\bar e_i,\bar e_j,\bar e_l,\bar e_m)\bar{\Theta}(\bar
e_i,\bar e_j,\bar e_k)
=\sum_{i,j=1}^7\big[R^+-d\bar T\big](\bar e_i,\bar
e_j,\bar e_l,\bar e_m)\bar{\Theta}(\bar e_i,\bar e_j,\bar e_k) \\[8pt]
\!&\! =-\sum_{i,j=1}^7d\bar T(\bar e_i,\bar e_j,\bar e_l,\bar e_m)\bar{\Theta}%
(\bar e_i,\bar e_j,\bar e_k),
\end{array}
\end{equation}
where we used the fact that the holonomy of $\nabla^+$ is contained in $G_2$%
, i.e. $\sum_{i,j=1}^7R^+(\bar e_i,\bar e_j,\bar e_l,\bar e_m)\bar{\Theta}%
(\bar e_i,\bar e_j,\bar e_k)=0$. Now, applying \eqref{torsion-general} and %
\eqref{g2f-general} we conclude that \eqref{inbar} is satisfied if and only
if \eqref{in11-general} holds.

\end{proof}

Let $\mathrm{D}_B$ be the $\nabla^{-}$ connection obtained by replacing $A$
with the matrix $B$ in Lemma \ref{nablamin-inst-general}, but allowing $B$
to be singular, $B\in {\mathfrak{g}\mathfrak{l}}_3(\mathbb{R})$.
Hence, the connection $\mathrm{D}_B$ is a $%
G_2$-instanton with respect to the $G_2$ structure defined by %
\eqref{g2f-general} iff the dilaton function satisfies
\begin{equation}  \label{insab7-general}
\triangle e^{2f}=-2 |B|^2.
\end{equation}

Equation \eqref{p1-general} shows that the difference between the first
Pontrjagin forms of $\nabla^-$ and $\mathrm{D}_B$ is given by the formula
\begin{equation}  \label{p11-general}
8\pi^2\Big(p_1(\nabla^-)-p_1(\mathrm{D}_B)\Big)= -3\Big(|A|^2-|B|^2 \Big) %
\left(\triangle e^{-2f}\right)\, {e}^{1234}.
\end{equation}
Therefore, recalling \eqref{torsion-general} and taking into account %
\eqref{p11-general}, the anomaly cancellation condition is
\begin{multline*}
d\bar T-\frac{\alpha^{\prime }}48\pi^2\Big(p_1(\nabla^-)-p_1(\mathrm{D}_B)%
\Big) =-\left [ \triangle e^{2f}+2|A|^2 -\frac{3}{4}\alpha^{\prime}\Big(%
|A|^2-|B|^2\Big)\left (\triangle e^{-2f}\right)\,\right] {e}^{1234}=0
\end{multline*}
coupled with \eqref{insab7-general}. Notice that at this point the analysis
can proceed exactly as in \cite[Section 5.2]{FIUVas}. As a result we obtain
the following results depending on the $|A|^2-|B|^2$ being zero or non-zero.

For $B=O$, where $O$ is the zero matrix in ${\mathfrak{g}\mathfrak{l}}_3(\mathbb{R})$, and a fixed $%
e\in \mathbb{R}^4$ we let
\begin{equation}  \label{fundsol}
e^{2f}= \frac {3\alpha^{\prime }}{4|x-e|^2}, \quad x\in \mathbb{R}^4.
\end{equation}
Using logarithmic radial coordinates near the singularity (as e.g. in \cite%
{CHS}) it follows that the $4-D$ metric induced on $\mathbb{R}^4$ is
actually complete. In fact, taking the singularity at the origin, in the
coordinate $t=\sqrt {3\alpha^{\prime }}/2 \, \ln \left( 4{|x|^2}/{%
3\alpha^{\prime }}\right)=-\sqrt{3\alpha^{\prime }}\, f$, we have
that the dilaton and the $4-D$ metric can be expressed as follows
\begin{equation*}
f=-t \sqrt {3\alpha^{\prime }}, \qquad \bar g_H=\sum_{i=1}^4
e^{2f}(e^i)^2=dt^2+3\alpha^{\prime }ds^2_3,
\end{equation*}
where $ds^2_3$ is the metric on the unit three-dimensional sphere in the
four dimensional Euclidean space. The completeness of the horizontal metric
implies that the metric $\bar g=\bar g_H +(e^5)^2+(e^6)^2+(e^7)^2$ is also
complete. Thus, we proved

\begin{thrm}
\label{t:main2} The non-compact complete simply connected manifold $(K_{A},%
\bar{\Theta},\nabla^-,\mathrm{D}_O,f)$ described above is a complete $G_2$
manifold which solves the Strominger system with non-constant dilaton $f$
determined by \eqref{fundsol}, non-zero flux $H=\bar T $ and non-flat
instanton $\mathrm{D}_O$ using the first Pontrjagin form of $\nabla^-$ and
positive $\alpha^{\prime }$. Furthermore, $(K_{A},\bar{\Theta},\nabla^-,%
\mathrm{D}_O,f)$ also solves the heterotic equations of motion \eqref{mot} up
to the first order of $\alpha^{\prime }$.
\end{thrm}

On the other hand, in the case $|A|^2=|B|^2\not=0$ the anomaly condition is
trivially satisfied for any $\alpha^{\prime }$, provided the torsion is
closed, see Lemma \ref{nablamin-inst-general}. In this case the solution is
given by the solutions of \eqref{in11-general}. Furthermore, both $%
\nabla^{-} $ and $\mathrm{D}_B$ are $G_2$-instantons. For example, a
particular solution is obtained by taking
\begin{equation*}
e^{2f}=\frac{|A|^2}{4}(1-|x|^2)
\end{equation*}
defined in the unit ball.

\section{Solutions through contractions}

In this section we consider appropriate contractions of the
quaternion Heisenberg algebra, the geometric structures, the partial
differential equations and their solutions found in
sections \ref{s:compact ex} and \ref{s:non-compact ex} in the $G_2$-heterotic case,
and we show that they
converge to the heterotic solutions on 6-dimensional inner non-K\"ahler
spaces constructed in \cite{FIUVas}.
Furthermore, this method allows us to find new heterotic solutions with non-constant dilaton in dimension 5.

\subsection{Six dimensional solutions}

\label{ss:6-D} Using the classification results of \cite{UV2} it was shown
in \cite{FIUVas}, that the 2-step nilmanifolds which are $\mathbb{T}^2$
bundles over $\mathbb{T}^4$ with connection 1-forms of anti-self-dual
curvature are precisely the invariant balanced Hermitian metrics with
Abelian complex structure $J$, i.e., $[JX,JY]=[X,Y]$. Moreover, in such case the
Lie algebra underlying $M$ is isomorphic to $\mathfrak{h}_3$ or $\mathfrak{h}%
_5$. Here, $\mathfrak{h}_3$ is the Lie algebra underlying the nilmanifold
given by the product of the 5-dimensional generalized Heisenberg nilmanifold
by $S^1$, while $\mathfrak{h}_5$ is the Lie algebra underlying the Iwasawa
manifold. The structure equations of the Lie algebra $\mathfrak{h}_5$ are
\begin{equation}  \label{fam1h5}
de^1 = de^2=de^3=de^4=0,\quad de^5 = b\,\sigma_2, \quad de^6 =
a\,\sigma_1-b\,\sigma_3,
\end{equation}
while $\mathfrak{h}_3$ is given by
\begin{equation*}
de^1 = de^2=de^3=de^4=0,\quad de^5 = 0, \quad de^6 = a\,\sigma_1,
\end{equation*}
where $a,b\in\mathbb{R}^*$ and $\sigma_i$ are the anti-self-dual forms
on $\mathbb{R}^4$, see after \eqref{ecus-general}. Clearly $\mathfrak{h}_3$ is a
contraction of $\mathfrak{h}_5$ and both are contractions of $\mathfrak{g(%
\mathbb{H})}$, see \eqref{ecus-general}.

It is a remarkable fact that the geometric structures, the partial
differential equations and their solutions found in sections \ref{s:compact
ex} and \ref{s:non-compact ex} converge to the heterotic solutions on
6-dimensional inner non-K\"ahler spaces found in \cite{FIUVas} as we explain
next in details for $\mathfrak{h}_5$. The $SU(3)$ structure and
corresponding solution based on $\mathfrak{h}_3$ is handled analogously.

Clearly $\mathfrak{h}_5$ is a contraction of $\mathfrak{K}_A$ when $%
\varepsilon\rightarrow 0$ using, for example,
\begin{equation*}
A_\varepsilon\overset{def}{=}\left(\!\!\!
\begin{array}{ccc}
0 & b & 0 \\
a & 0 & -b \\
0 & 0 & \varepsilon%
\end{array}
\!\right).
\end{equation*}
Notice that by \eqref{conf-general} we have $\bar
e^7=e^7_\varepsilon=\varepsilon\gamma^7\rightarrow 0$ as $\varepsilon\rightarrow 0$.
With the above choice of $A_\varepsilon$ we write the $G_2$-form %
\eqref{g2f-general} in the usual way as
\begin{equation*}
\bar\Theta_\varepsilon=\bar F \wedge e^7_\varepsilon +\bar\Psi^{+}, \qquad
\bar F=e^{2f}\omega_1+ e^{56}, \qquad \bar\Psi^{+}=e^{2f}
(\omega_2\wedge e^5-\omega_3\wedge e^6)
\end{equation*}
using \eqref{conf-general} and indicating with subscript $\varepsilon$ the
dependence on $\varepsilon$ through the matrix $A_\varepsilon$. In addition,
we let $\bar\Psi^{-}=e^{2f} (\omega_2\wedge e^6+\omega^3\wedge e^5)$. In the
limit $\varepsilon\rightarrow 0$, the forms $\bar F$, $%
\bar\Psi^{\pm} $ define an $SU(3)$ structure $(\bar F,\bar\Psi^{\pm})$ on a six
dimensional space, obtained through the ansatz proposed in \cite{GP} from a $%
\mathbb{T}^2$ bundle over $\mathbb{T}^4$ (corresponding to $f=0$), see \cite[Section 3.2]%
{FIUVas} for details in the case of $\mathfrak{h}_5$. Therefore, this $SU(3)$
structure solves the first two Killing spinor equations. Furthermore, the
Pontrjagin form of the $\nabla^{-}$ connection is given again by %
\eqref{p1-general} as shown in \cite[Section 3]{FIUVas}. In fact, the
connection forms \eqref{conf-general} and the corresponding curvature
2-forms (notice that $(\Omega^{-}_\varepsilon)^{\bar i}_{\bar 7}\rightarrow 0
$ for all $i$) converge to those of the $\nabla^{-}$ connection of the $SU(3)
$ case. Similarly, the seven dimensional anomaly cancellation conditions of
Sections \ref{s:compact ex} and \ref{s:non-compact ex} turn into the anomaly
cancellation conditions for the corresponding six dimensional structures. As
a consequence we obtain the six-dimensional solutions with non-constant
dilaton found in \cite{FIUVas}.

\subsection{Five dimensional solutions}

\label{ss:5-D} We begin with recalling the five dimensional Lie algebra ${%
\mathfrak{h}}(2,1)$ \cite{FIUVdim5} with structure equations
\begin{equation}  \label{diff-heisenberg}
d e^j =0,\ j=1,2,3,4,\qquad d e^5 =\sum_{i=1}^3 a_i\,\sigma_i, \quad a_i\in
\mathbb{R}, \quad (a_1,a_2,a_3)\not= (0,0,0).
\end{equation}
Without loss of generality we will suppose next that $a_1\not= 0$.
Clearly ${\mathfrak{h}}(2,1)$ is a contraction of $\mathfrak{K}_A$, see %
\eqref{ecus-general}, using, for example,
\begin{equation*}
A_\varepsilon\overset{def}{=}\left(\!\!\!
\begin{array}{ccc}
a_1 & a_2 & a_3 \\
0 & \varepsilon & 0 \\
0 & 0 & \varepsilon%
\end{array}
\!\right)
\end{equation*}
and letting $\varepsilon\rightarrow 0$. Notice that by \eqref{conf-general}
we have
\begin{equation}  \label{e:basis contraction 5D}
\bar e^i=e^i_\varepsilon=\varepsilon\gamma^i\rightarrow 0, \quad i=6,7
\end{equation}
when $\varepsilon\rightarrow 0$.

It was shown in \cite[Section 4]{FIUVdim5} that the $SU(2)$-structure $%
(e^5,\omega_1,\omega_2,\omega_3)$ is the unique family of left invariant
solutions (with constant dilaton) to the first two Killing spinor equations
on a five dimensional Lie group. Furthermore, \cite{FIUVdim5} continued on
showing that for $\nabla=\nabla^{+}$ or $\nabla=\nabla^{g}$ and suitably
defined instantons one can obtain compact (nilmanifolds) heterotic solutions
with constant dilaton. However, since the first Pontrjagin form of the
connection $\nabla^{-}$ vanishes there is no compact solution with constant dilaton to the
heterotic supersymmetry equations satisfying the anomaly cancellation
condition with $\nabla=\nabla^{-}$.

From the current point of view, we consider the case $\nabla=\nabla^{-}$ as
a contraction limit of the $G_2$-solutions in Sections \ref{s:compact ex}
and \ref{s:non-compact ex}. As a result we will obtain five dimensional solutions
with non-constant dilaton. Indeed,   applying \eqref{e:basis contraction 5D} and \eqref{diff-heisenberg} in \eqref{torg1-general} we obtain the expression for the  torsion in dimension five described in \eqref{tor5f}. In other words, the torsion in dimension five is obtained as a dimensional reduction of the torsion in dimension seven. Furthermore, the
Pontrjagin form of the $\nabla^{-}$ connection is given again by %
\eqref{p1-general} taking into account \eqref{diff-heisenberg}. In fact, the
connection forms \eqref{conf-general} and the corresponding curvature
2-forms (notice that $(\Omega^{-}_\varepsilon)^{\bar i}_{\bar 6}\rightarrow 0
$  and $(\Omega^{-}_\varepsilon)^{\bar i}_{\bar 7}\rightarrow 0
$ for all $i$) converge to those of the $\nabla^{-}$ connection of the $SU(2)
$ structure in dimension five. Similarly, the seven dimensional anomaly cancellation conditions of
Sections \ref{s:compact ex} and \ref{s:non-compact ex} turn into the anomaly
cancellation conditions for the corresponding five dimensional structures. At this point we turn to the construction of the five dimensional solutions with non-constant dilaton.

The five dimensional version of
Theorem \ref{t:cpmct strominger} is Theorem \ref{t:cpmct strominger5} below.  In the statement of Theorem \ref{t:cpmct strominger5} we use %
\eqref{conf 2forms}, i.e., $\bar\omega_i=e^{2f}\omega_i$, $i=1,2,3,4$. Let $H(2,1)$ be the five dimensional connected
simply connected Lie group $H(2,1)$ with Lie algebra ${\mathfrak{h}}(2,1)$,
We consider a lattice $\Gamma$ in the Lie group $H(2,1)$ with period $2 \tau_{+}$
in all variables, where $2\tau_{+}$ is the period of the Weirstrass' $\mathcal{P%
}$ function \eqref{e:Weirstrass eqn}.
The $SU(2)$ instanton  $D_\Lambda$ below corresponds to the
instanton obtained from the one in Proposition \ref{instanton2-general} by
setting the last two columns equal to zero or letting $\varepsilon%
\rightarrow 0$, see \eqref{e:basis contraction 5D}.

\begin{lemma}
\label{instanton2-general 5D} Let $\mathrm{D}_{\Lambda }$, $\Lambda
=(\lambda _{1},\lambda _{2},\lambda _{3})\in \mathbb{R}^{3}$, be the linear
connection on the Lie group ${H}(2,1)$ whose possibly non-zero 1-forms are
given as follows
\begin{equation*}
\begin{array}{l}
(\omega ^{\mathrm{D}_{\Lambda }})_{\bar{2}}^{\bar{1}}=-(\omega ^{\mathrm{D}%
_{\Lambda }})_{\bar{1}}^{\bar{2}}=-(\omega ^{\mathrm{D}_{\Lambda }})_{\bar{4}%
}^{\bar{3}}=(\omega ^{\mathrm{D}_{\Lambda }})_{\bar{3}}^{\bar{4}}=\lambda
_{1}\,\bar{e}^{5}, \\[8pt]
(\omega ^{\mathrm{D}_{\Lambda }})_{\bar{3}}^{\bar{1}}=-(\omega ^{\mathrm{D}%
_{\Lambda }})_{\bar{1}}^{\bar{3}}=(\omega ^{\mathrm{D}_{\Lambda }})_{\bar{4}%
}^{\bar{2}}=-(\omega ^{\mathrm{D}_{\Lambda }})_{\bar{2}}^{\bar{4}}=\lambda
_{2}\,\bar{e}^{5}, \\[8pt]
(\omega ^{\mathrm{D}_{\Lambda }})_{\bar{4}}^{\bar{1}}=-(\omega ^{\mathrm{D}%
_{\Lambda }})_{\bar{1}}^{\bar{4}}=-(\omega ^{\mathrm{D}_{\Lambda }})_{\bar{3}%
}^{\bar{2}}=(\omega ^{\mathrm{D}_{\Lambda }})_{\bar{2}}^{\bar{3}}=\lambda
_{3}\,\bar{e}^{5}.%
\end{array}%
\end{equation*}%
Then, $\mathrm{D}_{\Lambda }$ is an $SU(2)$-instanton with respect to the $%
SU(2)$ structure defined by $(e^5,\bar{\omega_1},\bar{\omega_2},\bar{\omega_3%
})$.
\end{lemma}
We skip the proof which is similar to the proof of Proposition \ref{instanton2-general}.
The five dimensional version of Theorem~\ref{t:cpmct strominger} follows.

\begin{thrm}
\label{t:cpmct strominger5} Let $(e^5,\bar{\omega_1},\bar{\omega_2},\bar{%
\omega_3})$ be the $SU(2)$ structure on the
Lie group $H(2,1)$.
The conformally compact five manifold $M^5=(\Gamma\backslash H(2,1),\eta^5,%
\bar{\omega_1},\bar{\omega_2},\bar{\omega_3},\nabla^-, D_{\Lambda}, f)$ is a
conformally quasi-Sasakian five manifold which solves the Strominger system
with non-constant dilaton $f$, non-trivial flux $H=\bar T$ and non-flat
instanton $D_{\Lambda}$ using the first Pontrjagin form of $\nabla^-$ and
negative $\alpha^{\prime }$. The dilaton $f$ depends on one variable and is
determined as a real slice of the Weierstrass' elliptic function. In
addition, $M^5$ satisfies the heterotic equations
of motion \eqref{mot} up to first order of $\alpha^{\prime }$.
\end{thrm}

In order to obtain the five dimensional version of Theorem~\ref{t:main2} we
use the following property of the $\nabla^{-}$ connection whose 1-forms are
\begin{equation*}
\begin{array}{l}
(\omega ^{-})_{\bar{2}}^{\bar{1}}=(\omega ^{-})_{\bar{4}}^{\bar{3}%
}=e^{-f}\left( f_{2}\,\bar{e}^{1}-f_{1}\,\bar{e}^{2}+f_{4}\,\bar{e}%
^{3}-f_{3}\,\bar{e}^{4}\right) ,\;\;\; \\
(\omega ^{-})_{\bar{3}}^{\bar{1}}=-(\omega ^{-})_{\bar{4}}^{\bar{2}%
}=e^{-f}\left( f_{3}\,\bar{e}^{1}-f_{4}\,\bar{e}^{2}-f_{1}\,\bar{e}%
^{3}+f_{2}\,\bar{e}^{4}\right) , \\
(\omega ^{-})_{\bar{4}}^{\bar{1}}=(\omega ^{-})_{\bar{3}}^{\bar{2}%
}=e^{-f}\left( f_{4}\,\bar{e}^{1}+f_{3}\,\bar{e}^{2}-f_{2}\,\bar{e}%
^{3}-f_{1}\,\bar{e}^{4}\right) ,\;\;\; \\[6pt]
(\omega ^{-})_{\bar{5}}^{\bar{1}}=e^{-2f}\left( -a_{11}\,\bar{e}^{2}-a_{12}\,%
\bar{e}^{3}-a_{13}\,\bar{e}^{4}\right) ,\;\;(\omega ^{-})_{\bar{5}}^{\bar{2}%
}=e^{-2f}\left( a_{11}\,\bar{e}^{1}+a_{13}\,\bar{e}^{3}-a_{12}\,\bar{e}%
^{4}\right) ,\;\;\; \\[6pt]
(\omega ^{-})_{\bar{5}}^{\bar{3}}=e^{-2f}\left( a_{12}\,\bar{e}^{1}-a_{13}\,%
\bar{e}^{2}+a_{11}\,\bar{e}^{4}\right) ,\;\;\;(\omega ^{-})_{\bar{5}}^{\bar{4%
}}=e^{-2f}\left( a_{13}\,\bar{e}^{1}+a_{12}\,\bar{e}^{2}-a_{11}\,\bar{e}%
^{3}\right),\;\;\;%
\end{array}
\end{equation*}
which are obtained from \eqref{connection-forms-general} taking into account \eqref{e:basis contraction 5D}.
\begin{lemma}
\label{nablamin-inst-general 5D} The $(-)$-connection of the $SU(2)$
structure $(e^5,\bar{\omega_1},\bar{\omega_2},\bar{\omega_3})$ is an $SU(2)$
instanton iff the torsion 3-form is closed, $d\bar T=0$, i.e., the dilaton
function $f$ satisfies equation \eqref{in11-general}.
\end{lemma}

The proof of Lemma \ref{nablamin-inst-general 5D} is very similar to the
proof of Lemma \ref{nablamin-inst-general} and involves a direct calculation.
Let $D_O$ be the $SU(2)$ instanton constructed by Lemma \ref%
{nablamin-inst-general 5D} in the case $A=O$-the zero:

\begin{thrm}
\label{t:main25} The non-compact simply connected five manifold $%
(H(2,1),e^5,\bar{\omega_1},\bar{\omega_2},\bar{\omega_3},\nabla^-,
D_O, f)$ is a complete conformally quasi-Sasakian five manifold which
solves the Strominger system with non-constant dilaton $f$ determined by %
\eqref{fundsol}, non-trivial flux $H=\bar T$, non-flat instanton $D_O$
using the first Pontrjagin form of $\nabla^-$ and positive $\alpha^{\prime }$%
.

The complete five manifold $(H(2,1),e^5,\bar{\omega_1},\bar{\omega_2},%
\bar{\omega_3},\nabla^-, D_O, f)$ satisfies the heterotic equations
of motion \eqref{mot} up to first order of $\alpha^{\prime }$.
\end{thrm}


\begin{thebibliography}{99}
\bibitem{Ahl} {\ L.V. Ahlfors, Complex analysis: An introduction of the
theory of analytic functions of one complex variable, 
McGraw-Hill Book Co., New York-Toronto-London 1966. }

\bibitem{Sharp} L. B. Anderson, J. Gray, E. Sharpe, \emph{Algebroids,
Heterotic Moduli Spaces and the Strominger System}, arXiv:1402.1532 [hep-th]

\bibitem{AG1} {\ B. Andreas, M. Garc\'{\i}a-Fern\'andez, \emph{Heterotic
non-K\"ahler geometries via polystable bundles on Calabi-Yau threefolds}, J.
Geom. Phys. \textbf{62} (2012), 183--188.}

\bibitem{AG2} {\ B. Andreas, M. Garc\'{\i}a-Fern\'andez, \emph{Solutions of
the Strominger system via stable bundles on Calabi-Yau threefolds}, Commun.
Math. Phys. \textbf{332} (2014), 1381--1383. }

\bibitem{AG3} {\ B. Andreas, M. Garc\'{\i}a-Fern\'andez, \emph{Note on solutions of the Strominger system from unitary
representations of cocompact lattices of $SL(2, \mathbb{C})$}, Commun. Math.
Phys.  \textbf{315} (2012), 153--168. }



\bibitem{BBDG} {\ K. Becker, M. Becker, K. Dasgupta, P.S. Green, \emph{%
Compactifications of heterotic theory on non-K\"ahler complex manifolds: I},
JHEP \textbf{0304} (2003), 007. }

\bibitem{BBE} {\ K. Becker, M. Becker, K. Dasgupta, P.S. Green, E. Sharpe,
\emph{Compactifications of heterotic strings on non-K\"ahler complex
manifolds: II}, Nuclear Phys. \textbf{B 678} (2004), 19--100. }

\bibitem{BBDP} {\ K. Becker, M. Becker, K. Dasgupta, S. Prokushkin, \emph{%
Properties of heterotic vacua from superpotentials}, Nuclear Phys. \textbf{B
666} (2003), 
144--174. } 

\bibitem{y4} {\ K. Becker, M. Becker, J.-X. Fu, L.-S. Tseng, S.-T. Yau,
\emph{Anomaly cancellation and smooth non-K\"ahler solutions in heterotic
string theory}, Nuclear Phys. \textbf{B 751} (2006), 108--128. }

\bibitem{BBCG} {\ K. Becker, C. Bertinato, Y.-C. Chung, G. Guo, \emph{%
Supersymmetry breaking, heterotic strings and fluxes}, Nuclear Phys. \textbf{%
B 823} (2009), 428--447. } 

\bibitem{BSethi} {\ K. Becker, S. Sethi, \emph{Torsional heterotic geometries%
}, Nuclear Phys. \textbf{B 820} (2009), 1--31. } 

\bibitem{BBW} I. Bena, N. Bobev, N. Warner, \emph{Bubles on manifolds with
U(1) isometry}, JHEP \textbf{0708} (2007), 004.

\bibitem{Berg} {\ E.A. Bergshoeff, M. de Roo, \emph{The quartic effective
action of the heterotic string and supersymmetry}, Nuclear Phys. \textbf{B
328} (1989), 439--468. }


\bibitem{Bis} {\ I. Biswas, A. Mukherjee, \emph{Solutions of Strominger
system from unitary representations of cocompact lattices of} SL(2,$\mathbb{C%
}$), Commun. Math. Phys. \textbf{322} (2013), 
373--384. } 

\bibitem{CHS} {\ C.G. Callan, J.A. Harvey, A. Strominger, \emph{Worldsheet
approach to heterotic instantons and solitons}, Nuclear Phys. \textbf{B 359}
(1991), 
611--634. }



\bibitem{Car1} {\ G.L. Cardoso, G. Curio, G. Dall'Agata, D. L\"ust, \emph{%
BPS action and superpotential for heterotic string compactifications with
fluxes}, JHEP \textbf{0310} (2003), 004. }

\bibitem{CCDLMZ} {\ G.L. Cardoso, G. Curio, G. Dall'Agata, D. L\"ust, P.
Manousselis, G. Zoupanos, \emph{Non-K\"ahler string backgrounds and their
five torsion classes}, Nuclear Phys. \textbf{B 652} (2003), 5--34. }


\bibitem{ConS} D. Conti, S. Salamon, \emph{Generalized Killing spinors in
dimension $5$}, Trans. Amer. Math. Soc. \textbf{359} (2007), 5319--5343.


\bibitem{DFG} {\ K. Dasgupta, H. Firouzjahi, R. Gwyn, \emph{On the warped
heterotic axion}, JHEP \textbf{0806} (2008), 056. }

\bibitem{Bwit} {\ B. de Wit, D.J. Smit, N.D. Hari Dass, \emph{Residual
supersymmetry of compactified D=10 supergravity}, Nuclear Phys. \textbf{B 283%
} (1987), 165--191. }



\bibitem{Erd} {\ A. Erd\'elyi, W. Magnus, F. Oberhettinger, F.G. Tricomi,
Higher transcendental functions. Vol. I, II and III,
Robert E. Krieger Publishing Co., Inc., Melbourne, Fla., 1981. }

\bibitem{FYau}  Teng Fei, Shing-Tung Yau, \emph{Invariant Solutions to the
Strominger System on Complex Lie Groups and Their Quotients},
arXiv:1407.7641.

\bibitem{FIUVdim5} {\ M. Fern\'{a}ndez, S. Ivanov, L. Ugarte, R. Villacampa,
\emph{Compact supersymmetric solutions of the heterotic equations of motion
in dimension 5}, Nuclear Phys. \textbf{B 820} (2009), 483--502. }

\bibitem{FIUV} {\ M. Fern\'andez, S. Ivanov, L. Ugarte, R. Villacampa, \emph{%
Non-K\"ahler heterotic-string compactifications with non-zero fluxes and
constant dilaton}, Commun. Math. Phys. \textbf{288} (2009), 677--697. }

\bibitem{FIUVdim7-8} {\ M. Fern\'{a}ndez, S. Ivanov, L. Ugarte, R.
Villacampa, \emph{Compact supersymmetric solutions of the heterotic
equations of motion in dimensions 7 and 8}, Adv. Theor. Math. Phys. \textbf{%
15} (2011), 245--284. }

\bibitem{FIUVas} {\ M. Fern\'andez, S. Ivanov, L. Ugarte, D. Vassilev, \emph{%
Non-K\"ahler heterotic string solutions with non-zero fluxes and
non-constant dilaton}, 
JHEP \textbf{1406} (2014), 073.}


\bibitem{FI1} Th. Friedrich, S. Ivanov, \emph{Parallel spinors and
connections with skew-symmetric torsion in string theory}, Asian J. Math.
\textbf{6} (2002), 303--336.

\bibitem{FI2} Th. Friedrich, S. Ivanov, \emph{Killing spinor equations in
dimension 7 and geometry of integrable $G_2$ manifolds}, J. Geom. Phys.,
\textbf{48} (2003), 1--11.

\bibitem{FLY} {\ J.-X. Fu, L.-S. Tseng, S.-T. Yau, \emph{Local heterotic
torsional models}, Commun. Math. Phys. \textbf{289} (2009), 1151--1169. }


\bibitem{y3} {\ J.-X. Fu, S.-T. Yau, \emph{The theory of superstring with
flux on non-K\"ahler manifolds and the complex Monge-Amp\`ere equation}, J.
Diff. Geom. \textbf{78} (2008), 369--428. }

\bibitem{GGMPR} J.P. Gauntlett, J. Gutowski, C. Hull, S. Pakis, H. Reall,
\emph{All supersymmetric solutions of minimal supergravity in five dimensions%
}, Class. Quantum Grav. \textbf{20} (2003), 4587--4634.

\bibitem{GKMW} {\ J.P. Gauntlett, N. Kim, D. Martelli, D. Waldram, \emph{%
Fivebranes wrapped on SLAG three-cycles and related geometry}, JHEP \textbf{%
0111} (2001), 018. }

\bibitem{GMPW} {\ J.P. Gauntlett, D. Martelli, S. Pakis, D. Waldram, \emph{$%
G $-structures and wrapped NS5-branes}, Commun. Math. Phys. \textbf{247}
(2004), 421--445. }

\bibitem{GMW} {\ J.P. Gauntlett, D. Martelli, D. Waldram, \emph{Superstrings
with intrinsic torsion}, Phys. Rev. \textbf{D 69} (2004), 086002, 27 pp. }

\bibitem{GKP} {\ S.B. Giddings, S. Kachru, J. Polchinski, \emph{Hierarchies
from fluxes in string compactifications}, Phys. Rev. \textbf{D 66} (2002),
106006, 16 pp. } 

\bibitem{GPap} {\ J. Gillard, G. Papadopoulos, D. Tsimpis, \emph{Anomaly,
fluxes and $(2,0)$ heterotic-string compactifications}, JHEP \textbf{0306}
(2003), 035. }

\bibitem{GP} {\ E. Goldstein, S. Prokushkin, \emph{Geometric model for
complex non-K\"ahler manifolds with $SU(3)$ structure}, Commun. Math. Phys.
\textbf{251} (2004), 65--78. }

\bibitem{GLP} {\ U. Gran, P. Lohrmann, G. Papadopoulos, \emph{The spinorial
geometry of supersymmetric heterotic string backgrounds}, JHEP \textbf{0602}
(2006), 063. }

\bibitem{GPRS} {\ U. Gran, G. Papadopoulos, D. Roest, P. Sloane, \emph{%
Geometry of all supersymmetric type I backgrounds}, JHEP \textbf{0708}
(2007), 074. }

\bibitem{GPR} {\ U. Gran, G. Papadopoulos, D. Roest, \emph{Supersymmetric
heterotic string backgrounds}, Phys. Lett. \textbf{B 656} (2007), 119--126. }

\bibitem{P} {\ U. Gran, G. Papadopoulos, \emph{Solution of heterotic Killing
spinor equations and special geometry}, Special metrics and supersymmetry,
144--161, AIP Conf. Proc., 1093, Amer. Inst. Phys., Melville, NY, 2009. }

\bibitem{GNic} M. G\"unaydin, H. Nikolai, \emph{Seven-dimensional octonionic
Yang-Mills instanton and its extension to an heterotic string soliton},
Phys. Lett. B \textbf{353} (1991), 169.



\bibitem{HP1} {\ P.S. Howe, G. Papadopoulos, \emph{Ultraviolet behavior of
two-dimensional supersymmetric non-linear sigma models}, Nuclear Phys.
\textbf{B 289} (1987), 264--276. }

\bibitem{Hu86} {\ C.M. Hull, \emph{Compactifications of the heterotic
superstring}, Phys. Lett. \textbf{B 178} (1986), 357--364. }

\bibitem{Hull} {\ C.M. Hull, \emph{Anomalies, ambiguities and superstrings},
Phys. Lett. \textbf{B 167} (1986), 51--55. }

\bibitem{HT} {\ C.M. Hull, P.K. Townsend, \emph{The two loop beta function
for sigma models with torsion}, Phys. Lett. \textbf{B 191} (1987), 115--121.
}

\bibitem{HuW} {\ C.M. Hull, E. Witten, \emph{Supersymmetric sigma models and
the heterotic string}, Phys. Lett. \textbf{B 160} (1985), 398--402. }

\bibitem{IMY} {\ H. Imazato, S. Mizoguchi, M. Yata, \emph{Taub-NUT crystal},
Int. J. Mod. Phys. A \textbf{26} (2011), 5143--5169. }


\bibitem{II} {\ P. Ivanov, S. Ivanov, \emph{$SU(3)$-instantons and $%
G_2,Spin(7) $-Heterotic string solitons}, Commun. Math. Phys. \textbf{259}
(2005), 79--102. }

\bibitem{Iv0} {\ S. Ivanov, \emph{Heterotic supersymmetry, anomaly
cancellation and equations of motion}, Phys. Lett. \textbf{B 685} (2010),
190--196. } 


\bibitem{IP1} {\ S. Ivanov, G. Papadopoulos, \emph{Vanishing theorems and
string backgrounds}, Class. Quantum Grav. \textbf{18} (2001), 1089--1110. }

\bibitem{IP2} {\ S. Ivanov, G. Papadopoulos, \emph{A no-go theorem for
string warped compactifications}, Phys. Lett. \textbf{B 497} (2001),
309--316. }



\bibitem{KY} {\ T. Kimura, P. Yi, \emph{Comments on heterotic flux
compactifications}, JHEP \textbf{0607} (2006), 030. }

\bibitem{KM} {\ T. Kimura, S. Mizoguchi, \emph{Chiral generations on
intersecting 5-branes in heterotic string theory }, JHEP \textbf{1004}
(2010), 028.} 

\bibitem{y1} {\ J. Li, S.-T. Yau, \emph{The existence of supersymmetric
string theory with torsion}, J. Diff. Geom. \textbf{70} 
(2005), 143--181. }

\bibitem{LV-P} H. Lu, J.F. Vazquez-Poritz, \emph{Resolution of overlaping
branes}, Phys. Lett. \textbf{B 534} (2002), 155.

\bibitem{MS} {\ D. Martelli, J. Sparks, \emph{Non-K\"ahler heterotic
rotations}, Adv. Theor. Math. Phys. \textbf{15} (2011), 131--174. }

\bibitem{MSeth}  Travis Maxfield, Savdeep Sethi, \emph{Domain Walls, Triples
and Acceleration}, arXiv:1404.2564.


\bibitem{Sethi} I. V. Melnikov, R. Minasian, S. Sethi, \emph{Heterotic
fluxes and supersymmetry}, arXiv:1403.4298 [hep-th].

\bibitem{MY} {\ S. Mizoguchi, M. Yata, \emph{Family unification via
quasi-Nambu-Goldstone fermions in string theory}, Prog. Theor. Exp. Phys.
(2013), 053B01. } 

\bibitem{Ossa} X. de la Ossa, E.E. Svanes, \emph{Holomorphic Bundles and
the Moduli Space of N=1 Heterotic Compactifications}, arXiv:1402.1725
[hep-th].

\bibitem{Ossa1}  Xenia de la Ossa, Eirik Eik Svanes, \emph{Connections,
Field Redefinitions and Heterotic Supergravity}, arXiv:1409.3347.

\bibitem{Pap} {\ G. Papadopoulos, \emph{New half supersymmetric solutions of
the heterotic string}, Class. Quantum Grav. \textbf{26} (2009) 135001, 26 pp.%
} 

\bibitem{Sen} {\ A. Sen, \emph{$(2,0)$ supersymmetry and space-time
supersymmetry in the heterotic string theory}, Nuclear Phys. \textbf{B 278}
(1986), 289--308. }


\bibitem{SM} S. Stotyn, R. Mann, \emph{Supergravity on an Atiyah-Hitchin base%
}, JHEP \textbf{0806} (2008), 087.

\bibitem{Str1} {\ A. Strominger, \emph{Heterotic solitons}, Nuclear Phys.
\textbf{B 343} (1990) 167--184. [Erratum-ibid. \textbf{B 353} (1991) 565.] }

\bibitem{Str} {\ A. Strominger, \emph{Superstrings with torsion}, Nuclear
Phys. \textbf{B 274} (1986), 253--284. }



\bibitem{UV2} {\ L. Ugarte, R. Villacampa, \emph{Balanced Hermitian geometry
on 6-dimensional nilmanifolds}, Forum Math. (to appear), arXiv:1104.5524v2
[math.DG]. }











\end{thebibliography}
\end{document}